\documentclass[12pt]{amsart}
\usepackage[a4paper,hmargin={2.5cm,2.5cm},vmargin={2.5cm,2.5cm}]{geometry}

\usepackage[utf8]{inputenc}
\usepackage{lipsum}
\usepackage[T1]{fontenc}
\usepackage{lmodern}
\usepackage{mathrsfs}
\usepackage[T1]{fontenc}
\usepackage[english]{babel}
\usepackage{csquotes}
\usepackage{hyperref}
\usepackage{cite}
\usepackage[shortlabels]{enumitem}
\usepackage{libertinus}
\usepackage{libertinust1math}
\usepackage{mathtools}
\mathtoolsset{showonlyrefs}

\title[Multiplicity and Regularity results for Quasilinear Elliptic Systems]{Multiplicity and Regularity Results for Quasilinear Elliptic Systems via Nonsmooth Critical Point Theory}
\date{\today}

 \author{Simone Mauro}
 \address{ Dipartimento di Matematica e Informatica, Università della Calabria, Ponte Pietro Bucci cubo 31B,   87036 ARCAVACATA DI RENDE, COSENZA, ITALY}
 \email{simone.mauro@unical.it}   

\numberwithin{equation}{section}

\usepackage[leqno]{amsmath}

\usepackage{tikz-cd}

\theoremstyle{plain}
\newtheorem{theorem}{Theorem}[section]
\newtheorem{lemma}[theorem]{Lemma}
\newtheorem{proposition}[theorem]{Proposition}
\newtheorem{corollary}[theorem]{Corollary}

\theoremstyle{plain}
\newtheorem{definition}[theorem]{Definition}

\newtheorem{example}{Example}[section]

\theoremstyle{plain}
\newtheorem{remark}[theorem]{Remark}

\RequirePackage{dsfont}
\newcommand{\N}{\field{N}}

\newcommand{\R}{\field{R}}

\newcommand{\field}[1]{\mathbb{#1}}
\newcommand{\wto}{\rightharpoonup}

\tikzset{%
    symbol/.style={%
        draw=none,
        every to/.append style={%
            edge node={node [sloped, allow upside down, auto=false]{$#1$}}}
    }
}
\usepackage{enumitem}

\tikzset{shorten <>/.style={shorten >=#1,shorten <=#1}}

\begin{document}
\begin{abstract}
    We study the quasilinear elliptic system
\[
-\textbf{div}(A(x,\boldsymbol u)|D\boldsymbol u|^{p-2}D\boldsymbol u)
+\frac{1}{p}\nabla_{\boldsymbol s}A(x,\boldsymbol u)|D\boldsymbol u|^p
= \boldsymbol g(x,\boldsymbol u)
\quad \text{in } \Omega,
\qquad \boldsymbol u = 0 \text{ on } \partial\Omega,
\]
where $p>1$, $\Omega\subset\mathbb R^N$ is a bounded domain with $N>1$, and $\boldsymbol g$ satisfies a subcritical growth condition. In this setting, the associated energy functional is, in general, neither differentiable nor locally Lipschitz in the natural Sobolev space.

By exploiting a nonsmooth critical point theory, we prove the existence of infinitely many weak solutions by means of an Equivariant Mountain Pass Theorem. In addition, we establish $L^\infty$-bounds for weak solutions by adapting a Moser-type iteration.

\noindent \textbf{Keywords:}  subcritical nonlinearities, vectorial $p$-Laplacian, $L^\infty$-estimates, quasilinear elliptic systems, nonsmooth critical point theory.\\
\noindent \textbf{2020 MSC:}  35A01, 35A15, 35J20, 35J25, 35J62.

\end{abstract}
\maketitle
\section{Introduction}

Quasilinear elliptic systems, driven by the $p$-Laplace operator
\[
-\boldsymbol\Delta_p\boldsymbol u:=-\textbf{div}(|D\boldsymbol u|^{p-2}D\boldsymbol u),\quad p>1,
\]
arise in several contexts in nonlinear analysis and continuum mechanics. A natural generalization consists in considering operators of the form
\[
-\textbf{div}(A(x,\boldsymbol u)|D\boldsymbol u|^{p-2}D\boldsymbol u)+\frac1p\nabla_{\boldsymbol s}A(x,\boldsymbol u)|D\boldsymbol u|^p,
\]
which lead to a wider class of quasilinear problems.

In this setting, however, a major difficulty appears: the associated energy functional is, in general, neither differentiable nor locally Lipschitz in the natural Sobolev space. As a consequence, classical variational methods cannot be directly applied, and one needs to resort to tools from nonsmooth critical point theory.

\subsection{Setting of the problem}
In order to study the problem, let $p>1$ and let $\Omega\subset\mathbb R^N$ be a bounded domain. 
We consider the energy functional $\mathcal J:W_0^{1,p}(\Omega)\to\mathbb R$ defined by
\[
\mathcal J(\boldsymbol u)=\frac1p\int_\Omega A(x,\boldsymbol u)|D\boldsymbol u|^p-\int_\Omega G(x,\boldsymbol u).
\]

We assume that $A:\Omega\times\R^m\to\R$ is a $C^1$-Carathéodory function satisfying, for a.e. $x\in\Omega$ and every $\boldsymbol s\in\R^m$,
\begin{equation}\tag{$a.1$}\label{a.1} 
|A(x,\boldsymbol s)|,\ |\nabla_{\boldsymbol s}A(x,\boldsymbol s)|\le C_0, 
\end{equation} 
\begin{equation}\tag{$a.2$}\label{a.2}
A(x,\boldsymbol s)\ge\nu_0>0,
\end{equation}
\begin{equation}\tag{a.3}\label{a.3} 
\boldsymbol s\cdot\nabla_{\boldsymbol s}A(x,\boldsymbol s)\ge0,
\end{equation}
Moreover, we assume that there exists a Lipschitz continuous function $\psi:\R\to\R$ such that
\begin{equation}\tag{$\psi$}\label{psi ipotesi}
   \sum_{j=1}^m\left\{ \sum_{i=1}^m\left[\frac{\sigma_i}{p}\partial_{s_i}A(x,\boldsymbol s)e^{\sigma_i\psi(s_i')-\sigma_i\psi(s_i)}\right]
   -A(x,\boldsymbol s)\psi'(s_j)e^{\sigma_j\psi(s_j')-\sigma_j\psi(s_j)}\right\}|\boldsymbol\xi_j|^2
    \le0,
\end{equation}
for every $\sigma_1,\dots,\sigma_m\in\{-1,1\}$, $\boldsymbol\xi_1,\dots,\boldsymbol \xi_m\in\R^N$, and $\boldsymbol s,\boldsymbol s'\in\R^m$ with $|\boldsymbol s-\boldsymbol s'|\le1$.

\begin{remark}
Assumption \eqref{psi ipotesi} improves hypothesis \textbf{A2} in \cite{SquassinaDiagonalSystems99}, where $\psi$ is required to be bounded. 
In the scalar case $m=1$, condition \eqref{psi ipotesi} reduces to
\[
\frac{\sigma}{p}\partial_sA(x,s)-A(x,s)\psi'(s)\le0,
\]
which is satisfied, for instance, by choosing $\psi(s)=\frac{C_0}{p\nu_0}s$. 

If $m>1$, the Lipschitz continuity of $\psi$ yields
\[
\psi(s_j')-\psi(s_j)\le|\psi'|_\infty\,|\boldsymbol s-\boldsymbol s'|\le|\psi'|_\infty,
\]
and therefore \eqref{psi ipotesi} holds provided that
\[
\psi'(s_j)\ge\frac{C_0m}{\nu_0p}e^{2|\psi'|_\infty}.
\]
In particular, if $\psi'\equiv C>0$, condition \eqref{psi ipotesi} is satisfied whenever
\[
\frac{C_0m}{\nu_0p}\le\frac{1}{2e}.
\]
\end{remark}

We further assume that $G:\Omega\times\R^m\to\R$ is a $C^1$-Carathéodory function such that
\[
G(x,\boldsymbol 0)=0, \qquad \nabla_{\boldsymbol s}G(x,\boldsymbol 0)=\boldsymbol 0,
\]
and whose gradient $\boldsymbol g:=\nabla_{\boldsymbol s}G$ satisfies the growth condition
\begin{equation}\tag{$g.1$}\label{g.1}
|\boldsymbol g(x,\boldsymbol s)|\le a(x)+b|\boldsymbol s|^{q-1},
\end{equation}
where $a\in L^r(\Omega),b\ge0$ and
\[
r \ge
\begin{cases}
\frac{Np}{Np-N+p} &\text{if } N> p,\\
1&\text{if } N\le p,
\end{cases}
\qquad
p<q<p^*:=
\begin{cases}
\frac{Np}{N-p} & \text{if } N > p, \\
+\infty & \text{if } N \le p.
\end{cases}
\]

When $N<p$, due to the embedding $W_0^{1,p}(\Omega)\hookrightarrow L^\infty(\Omega)$, the analysis becomes simpler. 
For this reason, in the sequel, we focus on the case $N\ge p$.

Finally, we assume a super-$p$-linear behavior at infinity. Namely, there exist $R>0$, $\mu>p$, and $\gamma\in(0,\mu-p)$ such that, for $|\boldsymbol s|\ge R$,
\begin{align}
0<\mu G(x,\boldsymbol s)&\le \boldsymbol g(x,\boldsymbol s)\cdot \boldsymbol s, \tag{$g.2$}\label{g.2}\\
\boldsymbol s\cdot\nabla_{\boldsymbol s}A(x,\boldsymbol s)&\le\gamma A(x,\boldsymbol s). \tag{$a.4$}\label{a.4}
\end{align}

Under the above assumptions, the functional $\mathcal J$ is well defined and continuous on $W_0^{1,p}(\Omega)$. Moreover,
\begin{align*}
    \langle \mathcal J'(\boldsymbol u),\boldsymbol v\rangle&=
    \int_\Omega A(x,\boldsymbol u)|D\boldsymbol u|^{p-2}D\boldsymbol u\cdot D\boldsymbol v+\frac1p\int_\Omega \left(\nabla_{\boldsymbol s}A(x,\boldsymbol u)\cdot \boldsymbol v\right)|D\boldsymbol u|^p
    -\int_\Omega\boldsymbol g(x,\boldsymbol u)\cdot\boldsymbol v,
\end{align*}
for every $\boldsymbol v\in W_0^{1,p}(\Omega)\cap L^\infty(\Omega)$.

This naturally leads to the following notion of weak solution.
\begin{definition}
    A function $\boldsymbol u\in W_0^{1,p}(\Omega)$ is a weak solution of
    \[
    \tag{$\mathcal P$}\label{P}
    \begin{cases}
        -\textbf{div}(A(x,\boldsymbol u)|D\boldsymbol u|^{p-2}D\boldsymbol u)+\frac1p\nabla_{\boldsymbol s}A(x,\boldsymbol u)|D\boldsymbol u|^p=\boldsymbol g(x,\boldsymbol u)&\text{in $\Omega$}\\
        \boldsymbol u=0&\text{on $\partial\Omega$,}
    \end{cases}
    \]
    if
    \[
    \int_\Omega A(x,\boldsymbol u)|D\boldsymbol u|^{p-2}D\boldsymbol u\cdot D\boldsymbol v+\frac1p\int_\Omega \left(\nabla_{\boldsymbol s}A(x,\boldsymbol u)\cdot \boldsymbol v\right)|D\boldsymbol u|^p=\int_\Omega\boldsymbol g(x,\boldsymbol u)\cdot\boldsymbol v,
    \]
    for every $\boldsymbol v\in C_c^\infty(\Omega)$.
\end{definition}

By a standard density argument, the above identity extends to all test functions $\boldsymbol v\in W_0^{1,p}(\Omega)\cap L^\infty(\Omega)$.

Observe that, when $A\equiv 1$, problem \eqref{P} reduces to the classical vectorial $p$-Laplacian
\[
\begin{cases}
-\boldsymbol\Delta_p\boldsymbol u=\boldsymbol g(x,\boldsymbol u)
& \text{in }\Omega,\\
\boldsymbol u=0 & \text{on }\partial\Omega,
\end{cases}
\]
since $\nabla_{\boldsymbol s}A\equiv 0$.

We now provide an example of a function $A(x,\boldsymbol s)$ satisfying the above assumptions.

\begin{example}
Let $\varepsilon>0$ and define
\[
A(x,\boldsymbol s)=A(\boldsymbol s)
:=1+\varepsilon\frac{|\boldsymbol s|^{p}}{1+|\boldsymbol s|^{p}}
=\frac{1+(1+\varepsilon)|\boldsymbol s|^p}{1+|\boldsymbol s|^p}.
\]
Then
\[
1\le A(\boldsymbol s)<1+\varepsilon 
\qquad \text{for every }\boldsymbol s\in\R^m.
\]

For $\boldsymbol s\neq 0$ we compute
\[
\nabla_{\boldsymbol s}A(\boldsymbol s)
=\varepsilon p\frac{|\boldsymbol s|^{p-2}}{(1+|\boldsymbol s|^p)^2}\boldsymbol s,
\]
and therefore
\[
|\nabla_{\boldsymbol s}A(\boldsymbol s)|
\le
\varepsilon p
\max_{\rho>0}\frac{\rho^{p-1}}{(1+\rho^p)^2}
=
\varepsilon \frac{(p+1)^2}{4p}
\left(\frac{p-1}{p+1}\right)^{\frac{p-1}{p}}.
\]
Moreover,
\[
\nabla_{\boldsymbol s}A(\boldsymbol s)\cdot\boldsymbol s
=
\frac{\varepsilon p\,|\boldsymbol s|^{p}}{(1+|\boldsymbol s|^{p})^2}
\ge 0
\quad \text{for every }\boldsymbol s\in\R^m.
\]
Hence \eqref{a.1}--\eqref{a.3} hold with 
\[
\nu_0=1
\quad \text{and} \quad
C_0=
\max\left\{
1+\varepsilon,\,
\varepsilon \frac{(p+1)^2}{4p}
\left(\frac{p-1}{p+1}\right)^{\frac{p-1}{p}}
\right\}.
\]

Furthermore, for $|\boldsymbol s|\ge R$ we have
\[
\frac{\nabla_{\boldsymbol s}A(\boldsymbol s)\cdot\boldsymbol s}{A(\boldsymbol s)}
=
\frac{\varepsilon p\,|\boldsymbol s|^{p}}
{(1+|\boldsymbol s|^{p})(1+(1+\varepsilon)|\boldsymbol s|^p)}
\le
\frac{\varepsilon p}{1+|\boldsymbol s|^{p}}
\le
\frac{\varepsilon p}{1+R^{p}}.
\]
Thus, given any $\gamma>0$, choosing $R>0$ sufficiently large so that
\[
\frac{\varepsilon p}{1+R^{p}}\le \gamma,
\]
condition \eqref{a.4} follows.

Finally, since 
\[
\frac{\nabla_{\boldsymbol s}A(\boldsymbol s)\cdot\boldsymbol s}
{A(\boldsymbol s)}
\le \varepsilon p,
\]
we see that \eqref{psi ipotesi} is satisfied provided
\[
\psi'(s_j)\ge \varepsilon m e^{2|\psi'|_\infty}
\quad \text{for every }
s_j\in\R,\ j=1,\dots,m.
\]
In particular, choosing $\psi(t)=\tfrac{t}{2}$ and assuming
\[
\varepsilon \le \frac{1}{2 e m},
\]
condition \eqref{psi ipotesi} follows immediately.
\end{example}

\subsection{State of the art}
The main aim of this paper is to investigate the existence of critical points of $\mathcal J$ within the framework of nonsmooth critical point theory \cite{nonsmooththeory1, degiovanni1994critical}. This approach has proven to be particularly effective in dealing with quasilinear problems where the associated energy functional is not differentiable in the classical sense, and has been successfully developed in several contributions for both scalar and vectorial equations \cite{caninoquasilineare1, caninoserdica, nonsmooththeory1, CaninoMauroNeumann2026, pellacci1997critical, squassina2009existence}.

In the \emph{scalar setting} (m=1), the existence of infinitely many solutions for Dirichlet problems of the type
\[
-\operatorname{div}(\mathcal A(x,u)\nabla u)+\frac12\, D_s \mathcal A(x,u)\nabla u\cdot\nabla u = g(x,u), 
\qquad 
u\in W_0^{1,2}(\Omega),
\]
where $\mathcal A(x,s)$ is a matrix with $C^1$-Carathéodory coefficients, was established in \cite{caninoquasilineare1, caninoserdica, nonsmooththeory1} by exploiting nonsmooth variational techniques.

 These results have later been extended to more general quasilinear equations of the form
\[
-\operatorname{div}(\mathscr A(x,u,\nabla u)) + \mathscr B(x,u,\nabla u) = 0\quad \text{in }\Omega, 
\qquad 
u\in W_0^{1,p}(\Omega),
\]
as shown in \cite{pellacci1997critical}, where multiplicity results are obtained within the same nonsmooth framework.

Related approaches can be found in \cite{arcoyaboccardo1, candela2009some, candela2009infinitely}, where the functional is restricted to the subspace $W_0^{1,p}(\Omega)\cap L^\infty(\Omega)$ in order to recover differentiability; however, this choice leads to additional difficulties in the verification of the Palais--Smale condition.

In the \emph{vectorial case}, the literature is more recent and still less developed. In \cite{canino2025vectorial}, existence results are obtained for every $p>1$, while $L^\infty$-estimates for weak solutions are proved for $p\ge2$, in the case of constant coefficient $A(x,\boldsymbol s)=1$.

The case of nonconstant coefficients $A(x,\boldsymbol s)$ has been addressed in \cite{arioli2000existence, SquassinaDiagonalSystems99, squassina2009existence, struwe1983quasilinear}, where nonsmooth critical point theory is applied to the vector-valued problem:
\[
-\operatorname{div}(\mathcal A_k(x,\boldsymbol u)\nabla u_k)
   + \frac12 \sum_{i=1}^m D_{s_k}\mathcal A_i(x,\boldsymbol u)\,\nabla u_i\cdot\nabla u_i
   = g_k(x,\boldsymbol u), 
   \]
with $k=1,\dots,m,$ although restricted to the case $p=2$.

\subsection{Main result}

\begin{theorem}\label{main thm multiplicity}
 Let $p>1$ and suppose that $A(x,\boldsymbol s)=A(x,-\boldsymbol s), G(x,\boldsymbol s)=G(x,-\boldsymbol s)$.   Assume also that \eqref{a.1}-\eqref{a.4}, \eqref{psi ipotesi}, and \eqref{g.1}-\eqref{g.2} hold. 
 \begin{itemize}
\item[$(i)$] There exists a sequence \(\{\boldsymbol u_n\} \subset W_0^{1,p}(\Omega)\) of weak solutions to \eqref{P} such that
\[
\mathcal J(\boldsymbol u_n)\to+\infty,\qquad\text{as $n\to+\infty$.}
\]
\item[$(ii)$]
 If  \eqref{g.1} is satisfied for $r\ge\frac Np$, any weak solution of \eqref{P} is also bounded. 
 In particular, $\boldsymbol u_n\in W_0^{1,p}(\Omega)\cap L^\infty(\Omega)$ for every $n\ge1$.
 \end{itemize}
\end{theorem}

Let $\boldsymbol u \in W_0^{1,p}(\Omega;\mathbb R^m)$ be a weak solution to
\[
\begin{cases}
-\boldsymbol\Delta_p \boldsymbol u = \boldsymbol g(x,\boldsymbol u) & \text{in } \Omega,\\
\boldsymbol u = 0 & \text{on } \partial\Omega.
\end{cases}
\]
As a consequence of Theorem~\ref{main thm multiplicity}-$(ii)$ (see also Corollary~\ref{regularity pLaplacian}) and of the results in~\cite{MingioneSteinTheorem}, if
\[
|\boldsymbol g(x,\boldsymbol s)| \le a(x) + b|\boldsymbol s|^{q-1},
\]
with $a \in L^r(\Omega)$ for some $r > N$ and $p < q < p^*$, then $\boldsymbol u \in C^1(\Omega)$.

Moreover, if in addition $\partial\Omega \in W^{2}L^{N-1,1}$\footnote{$\partial\Omega$ is locally the graph of a function whose second derivatives belong to the Lorentz space $L^{N-1,1}$. Recall that
\(
L^t \subset L^{N-1,1} \subset L^{N-1}
\quad \text{for every } t > N-1.
\)
}, then $|D\boldsymbol u| \in L^\infty(\Omega)$, and therefore $\boldsymbol u$ is Lipschitz continuous in $\Omega$, by~\cite{CianchiArma2014}. 

Finally, if $\partial\Omega \in C^{1,\alpha}$ for some $\alpha \in (0,1)$ and there exists a constant $K_0 > 0$ such that
\[
|\boldsymbol g(x,\boldsymbol s)| \le K_0\bigl(1 + |\boldsymbol s|^{q-1}\bigr),
\qquad p < q < p^*,
\]
then $\boldsymbol u \in C^{1,\beta}(\overline\Omega)$ for some $\beta \in (0,1)$ by \cite{benedetto1989boundary}.

We also recall that the regularity theory for quasilinear elliptic systems is by now well developed in the case of source terms independent of the unknown, namely when $\boldsymbol g(x,\boldsymbol u)=\boldsymbol g(x)$; we refer to 
\cite{balci2022pointwise,benedetto1989boundary,CianchiArma2014,cianchi2019optimal,MingioneSteinTheorem,montoro2025regularity,SchmidtMinimizer14,sciunzi2025global}
and the references therein.

\begin{remark}
    Let $N>2$, $p=2, m=1$ and consider \eqref{P} with $\Omega=B_1(0)$, $ g(x,s)=0$ and 
    \[
    A(x,s)=1+\frac{1}{|x|^{12(N-2)}e^s+1}.
    \]
    Then $u(x)=-12(N-2)\log|x|\in W_0^{1,2}(\Omega)$ is a weak solution of 
    \[
    \begin{cases}
    -\text{div}(A(x,u)\nabla u)+\frac12 D_sA(x,u)|\nabla u|^2=0&\text{in $\Omega$}\\
    u=0&\text{on $\partial\Omega$},
    \end{cases}
    \]
     but
     \[D_sA(x,s)=-\frac{|x|^{12(N-2)}e^s}{(|x|^{12(N-2)}e^s+1)^2}\le0,\]
   and  $u\notin L^\infty(\Omega)$, even though $u\in L^t(\Omega)$ for every $t\ne\infty$. In particular, singular solutions may exist if the assumption \eqref{a.3}
    is not satisfied.
\end{remark}

\subsection{Structure of the paper} 
We conclude this introduction by briefly describing the structure of the paper.
In Section 2 we recall some basic tools from nonsmooth critical point theory and establish preliminary results concerning the functional $\mathcal J$, including properties of the weak slope and suitable compactness conditions. 
In Section 3 we prove \textit{a priori} $L^\infty$-bounds for weak solutions by means of a Moser-type iteration argument. 
Section 4 is devoted to the analysis of suitable compactness properties, where we establish strong convergence results for appropriate sequences. 
Finally, in Section 5 we combine the previous results to prove the multiplicity of weak solutions, exploiting an Equivariant Mountain Pass Theorem.

\subsection{Notation}

We set
\(
W_0^{1,p}(\Omega;\mathbb{R}^m)
:= \big( W_0^{1,p}(\Omega;\mathbb{R}) \big)^m,
\)
and write
\(
W^{-1,p'}(\Omega;\mathbb{R}^m)
\)
for its dual space. Whenever no confusion arises, we simply write
\(
W_0^{1,p}(\Omega)
\quad \text{and} \quad
W^{-1,p'}(\Omega).
\)
The same convention applies to $C_c^\infty(\Omega)$.

If $p>1$ and $q\in[1,\infty]$, we use the shorthand
\[
\|\boldsymbol u\|:=\|\boldsymbol u\|_{W_0^{1,p}(\Omega)}, 
\qquad 
|\boldsymbol u|_q:=\|\boldsymbol u\|_{L^q(\Omega)}.
\]

Given a function $\boldsymbol u \in W_0^{1,p}(\Omega)$, we denote by $D\boldsymbol u$ its Jacobian matrix, and we use the symbol "$\ \cdot\ $" to denote the scalar product both between real-valued matrices and between vectors in $\mathbb{R}^N$. In particular, for every $\boldsymbol u,\boldsymbol v\in W_0^{1,p}(\Omega)$, we have
\[
D\boldsymbol u\cdot D\boldsymbol v=\sum_{i=1}^m\nabla u_i\cdot\nabla v_i=\sum_{i=1}^m\sum_{j=1}^N\frac{\partial u_i}{\partial x_j}\ \frac{\partial v_i}{\partial x_j}.
\]

For every $(\boldsymbol v,\boldsymbol w)\in\R^{m+N}$, $\boldsymbol v\otimes \boldsymbol w\in \R^{m\times N}$ is the matrix with coefficients
\[
(\boldsymbol v\otimes \boldsymbol w)_{ij}=v_iw_j,\qquad \text{for every $i=1,\dots,m, \ j=1,\dots,N$.}
\]

\section{Preliminaries}
We recall some results about the critical point theory of continuous functionals, developed in \cite{nonsmooththeory1, corvellec1993deformation,degiovanni1994critical}. In this setting, we consider $(X,d)$ a metric space and $f: X \to \R$ a continuous functional.

\begin{definition} \label{definition 1}
Let $(X,d)$ be a metric space and let $f: X \to \mathbb{R}$ be a continuous function. We consider $\sigma \ge 0$ such that there exist $\delta > 0$ and a continuous map $\mathscr{H}: B_{\delta}(u) \times [0, \delta] \to X$ such that
\begin{align*}
d(\mathscr{H}(v, t), v) \le t, \qquad
f(\mathscr{H}(v, t)) \le f(v) - \sigma t.
\end{align*}
  We denote with $|df|(u)$ the supremum of all such $\sigma$ and we call it the \emph{weak slope} of $f$ at $u$.
\end{definition}

\begin{theorem}[{\cite[Theorem 1.1.2]{nonsmooththeory1}}]\label{theorem 1}
Let $E$ be a normed space and $X \subset E$ an open subset. Let us fix $u \in X$ and $v \in E$ with $\|v\| = 1$. For each $w \in X$ we define
$$\overline{D}_+f(w)[v] := \limsup_{t \to 0^+} \frac{f(w + tv) - f(w)}{t}.$$
Then $|df|(u) \ge -\limsup_{w \to u} \overline{D}_+f(w)[v]$.
\end{theorem}

\begin{definition}\label{def p.critical}
Let $(X,d)$ be a metric space and let $f: X \to \mathbb{R}$ be continuous. We will say that $u \in X$ is a (lower) critical point if $|df|(u) = 0$. A (lower) critical point is said to be at the level $c \in \mathbb{R}$ if it is also true that $f(u) = c$.
\end{definition}

\begin{definition}\label{ps}
Let $(X,d)$ be a metric space and let $f:X\to\R$ be a continuous functional. A sequence $\{u_n\}\subset X$ is said a $(PS)_c$-sequence if
\begin{align*}
f(u_n)\to c,\qquad
|df|(u_n)\to0.
\end{align*}
Furthermore, we say that $f$ satisfies the $(PS)_c$-condition if every subsequence admits a convergent subsequence in $X$. 
If the $(PS)_c$-condition holds for every $c\in\R$ we say that $f$ satisfies the $(PS)$-condition.
\end{definition}
\begin{theorem}[Equivariant Mountain Pass, {\cite[Theorem 1.3.3]{nonsmooththeory1}}]\label{MPequi}	\hfill\\
Let $X$ be a Banach space and let $f:X\to\R$ be a continuous even functional. Suppose that
\begin{itemize}
\item $\exists\ \rho>0,\alpha>f(0)$ and a subspace $W\subset X$ of finite codimension such that $f\ge\alpha$ on  $\partial B_{\rho}\cap W$,
\item for every subspace $V$ of dimension $k$, there exists $R>0$ such that $f\le f(0)$ in $B_R^c\cap V$.
\end{itemize}
If $f$ satisfies the $(PS)_c$-condition for every $c\ge\alpha$, then there exists  a divergent sequence of critical values, namely there exists a sequence of critical points $\{u_n\}\subset X$ such that $c_n:=f(u_n)\to+\infty$.
\end{theorem}

\begin{proposition}
Assume that \eqref{a.1} and \eqref{g.1} hold.
The functional $\mathcal J$ is continuous in $W_0^{1,p}(\Omega)$.
Furthermore,
\begin{equation}\label{dis weak slope}
  |d\mathcal J|(\boldsymbol u)\ge\sup_{\substack{\boldsymbol v\in C_c^\infty(\Omega)\\\|\boldsymbol v\|_{W_0^{1,p}}\le1}}\langle \mathcal J'(\boldsymbol u),\boldsymbol v\rangle,
\end{equation}
where
\[
    \langle \mathcal J'(\boldsymbol u),\boldsymbol v\rangle=
    \int_\Omega A(x,\boldsymbol u)|D\boldsymbol u|^{p-2}D\boldsymbol u\cdot D\boldsymbol v+\frac1p\int_\Omega \left(\nabla_{\boldsymbol s}A(x,\boldsymbol u)\cdot \boldsymbol v\right)|D\boldsymbol u|^p-\int_\Omega\boldsymbol g(x,\boldsymbol u)\cdot\boldsymbol v.
\]
\end{proposition}
\begin{proof} 
It is standard to verify that $\mathcal J$ is continuous and that 
 for every $\boldsymbol u\in W_0^{1,p}(\Omega)$ and for every $\boldsymbol v\in W_0^{1,p}(\Omega)\cap L^\infty(\Omega)$,
    \begin{align*}
    \langle \mathcal J'(\boldsymbol u),\boldsymbol v\rangle
    &=\lim_{t\to0}\frac{\mathcal J(\boldsymbol u+t\boldsymbol v)-\mathcal J(\boldsymbol u)}{t}\\
    &=
    \int_\Omega A(x,\boldsymbol u)|D\boldsymbol u|^{p-2}D\boldsymbol u\cdot D\boldsymbol v+\frac1p\int_\Omega \left(\nabla_{\boldsymbol s}A(x,\boldsymbol u)\cdot \boldsymbol v\right)|D\boldsymbol u|^p-\int_\Omega\boldsymbol g(x,\boldsymbol u)\cdot\boldsymbol v.
    \end{align*}

 Notice that for every $\boldsymbol v\in C_c^\infty(\Omega)$, the map
\[
W_0^{1,p}(\Omega)\ni\boldsymbol u\mapsto\langle\mathcal J'(\boldsymbol u),\boldsymbol v\rangle\quad\text{is continuous.}
\]
Thus, according to Theorem \ref{theorem 1}, we have that
\[
\|\boldsymbol v\|_{W_0^{1,p}}\cdot|d\mathcal J|(\boldsymbol u)\ge-\langle\mathcal J'(\boldsymbol u),\boldsymbol v\rangle=\langle\mathcal J'(\boldsymbol u),-\boldsymbol v\rangle,
\]
for every $\boldsymbol v\in C_c^\infty(\Omega)$ with $\|\boldsymbol v\|_{W_0^{1,p}}\le1$. In particular,
\[
|d\mathcal J(\boldsymbol u)|\ge\sup_{\substack{\boldsymbol v\in C_c^\infty(\Omega)\\\|\boldsymbol v\|_{W_0^{1,p}}\le1}}\langle\mathcal J'(\boldsymbol u),\boldsymbol v\rangle.\qedhere
\]
\end{proof}
\begin{definition}
    Let $\{\boldsymbol u_n\}\subset W_0^{1,p}(\Omega)$ be a sequence. We say that $\{\boldsymbol u_n\}$ is a Concrete Palais-Smale sequence at level $c\in\R$ if $\mathcal J(\boldsymbol u_n)\to c$ and
    \[
  \boldsymbol  \omega_n:=-\boldsymbol{div}(A(x,\boldsymbol u_n)|D\boldsymbol u_n|^{p-2}D\boldsymbol u_n)+\frac1p\nabla_{\boldsymbol s}A(x,\boldsymbol u_n)|D\boldsymbol u_n|^p-\boldsymbol g(x,\boldsymbol u_n)\to0\quad \text{in $W^{-1,p'}(\Omega)$.}
    \]
    The functional $\mathcal J$ satisfies the Concrete Palais-Smale condition at level $c\in\R$ if every Concrete Palais-Smale sequence at level $c\in\R$ admits a convergent subsequence in $W_0^{1,p}(\Omega)$.
\end{definition}

\begin{proposition}\label{PS e CPS}
Let $c\in\R$ be a real number.
The following facts hold:
\begin{enumerate}
\item[$(i)$] if $\boldsymbol u$ is a (lower) critical point for $\mathcal J$, then $u$ is a weak solution of the problem \eqref{P},
\item[$(ii)$] each $(PS)_c$-sequence is also a $(CPS)_c$-sequence,
\item[$(iii)$] $\mathcal J$ satisfies $(CPS)_c$-condition$\implies \mathcal J$ satisfies $(PS)_c$-condition.
\end{enumerate}
\end{proposition}
\begin{proof}   
$(i)$\ Let  $\boldsymbol u \in W_0^{1,p}(\Omega) $
 be a (lower) critical point of  $\mathcal J.$
From \eqref{dis weak slope} we have
\[
\langle \mathcal J'(\boldsymbol u), \boldsymbol v \rangle 
\le |d\mathcal J|(\boldsymbol u) = 0,
\qquad 
\text{for every } \boldsymbol v \in C_c^\infty(\Omega)
\text{ with } \|\boldsymbol v\|_{W_0^{1,p}} \le 1.
\]
Since we may replace $\boldsymbol v$ by $-\boldsymbol v$, it follows that
\[
\langle \mathcal J'(\boldsymbol u), \boldsymbol v \rangle = 0.
\]
This proves the claim.

$(ii)$ Let $\{\boldsymbol u_n\}\subset W_0^{1,p}(\Omega)$ be a Palais-Smale sequence for $\mathcal J$. Arguing as in $(i)$, we deduce that
\[
\langle \mathcal J'(\boldsymbol u_n), \boldsymbol v \rangle 
= o(1),
\qquad 
\text{for every } \boldsymbol v \in C_c^\infty(\Omega)
\text{ with } \|\boldsymbol v\|_{W_0^{1,p}} \le 1,
\]
where $o(1) \to 0$ as $n \to +\infty$. Thus,
\[
\boldsymbol \omega_n 
:= -\boldsymbol{\operatorname{div}}\!\left(
A(x,\boldsymbol u_n)|D\boldsymbol u_n|^{p-2}D\boldsymbol u_n
\right)
+ \frac{1}{p}\nabla_{\boldsymbol s}A(x,\boldsymbol u_n)|D\boldsymbol u_n|^p
- \boldsymbol g(x,\boldsymbol u_n)
\to 0
\quad \text{in } W^{-1,p'}(\Omega),
\]
and $\{\boldsymbol u_n\}$ is a Concrete Palais-Smale sequence.

$(iii)$ Let $\{\boldsymbol u_n\} \subset W_0^{1,p}(\Omega)$ 
be a $(PS)_c$-sequence for $\mathcal J$. 
By $(ii)$, $\{\boldsymbol u_n\}$ is also a $(CPS)_c$-sequence for $\mathcal J$. 
Since $\mathcal J$ satisfies the $(CPS)_c$-condition, 
there exists a subsequence $\{\boldsymbol u_{n_k}\} \subset \{\boldsymbol u_n\}$ 
such that $\boldsymbol u_{n_k}$ converges strongly in $W_0^{1,p}(\Omega)$. 
Hence, the $(PS)_c$-condition holds.
\end{proof}

\begin{theorem}\label{brezis browder}
    Assume that \eqref{a.1} holds.
    Let $\boldsymbol u\in W_0^{1,p}(\Omega)$ and $\boldsymbol \omega\in W^{-1,p'}(\Omega)$ such that
    \begin{equation}\label{eq thm density}
    -\textbf{div}(A(x,\boldsymbol u)|D\boldsymbol u|^{p-2}D\boldsymbol u)+\frac1p\nabla_{\boldsymbol s}A(x,\boldsymbol u)|D\boldsymbol u|^p=\boldsymbol \omega\qquad\text{in $\mathcal D'(\Omega)$}.
    \end{equation}
    Suppose that there exists $\boldsymbol v\in W_0^{1,p}(\Omega)$ such that 
    \[
    \min\left\{\nabla_{\boldsymbol s}A(x,\boldsymbol u)\cdot\boldsymbol v,0\right\}\in L^1(\Omega).
    \]
    Then 
    \[
    \int_\Omega A(x,\boldsymbol u)|D\boldsymbol u|^{p-2}D\boldsymbol u\cdot D\boldsymbol v+\frac1p\int_\Omega\left(
    \nabla_{\boldsymbol s}A(x,\boldsymbol u)\cdot\boldsymbol v
    \right)|D\boldsymbol u|^p=\langle\boldsymbol\omega,\boldsymbol v\rangle.
    \]
\end{theorem}
\begin{proof}
We define $\boldsymbol T=(T_1,\dots,T_m):C_c^\infty(\Omega)\to\R$ as follows
\[
\langle \boldsymbol T,\boldsymbol v\rangle=\frac1p\int_\Omega\left( \nabla_{\boldsymbol s}A(x,\boldsymbol u)\cdot\boldsymbol v\right)|D\boldsymbol u|^p=\frac1p\sum_{i=1}^m\int_\Omega\left( D_{s_i}A(x,\boldsymbol u)v_i\right)|D\boldsymbol u|^p.
\]
Since $\boldsymbol u\in W_0^{1,p}(\Omega)$ solves \eqref{eq thm density}, we have that $\boldsymbol T\in W^{-1,p'}(\Omega)\cap L^1(\Omega;\R^m)$, where
\[
\boldsymbol T(x)=\frac1p\nabla_{\boldsymbol s}A(x,\boldsymbol u)|D\boldsymbol u|^p
\] 
is the representative in $L^1(\Omega;\R^m)$.
Furthermore, we have that \[
\boldsymbol T(x)\cdot\boldsymbol v(x)\ge\min\left\{\boldsymbol T(x)\cdot\boldsymbol v(x),0\right\}\in L^1(\Omega)\qquad \text{for a.e. $x\in\Omega$.}
\]
Thus, the thesis follows by applying \cite[Theorem 5]{brezisbrowder2}.
\end{proof}

\section{\texorpdfstring{$L^\infty$-estimates}{Linfty-estimates}}

\begin{theorem}\label{regularity}
Assume that \eqref{a.1}-\eqref{a.3} hold.   Suppose also that {$p>1$}, \eqref{g.1} holds with {$r\ge\frac Np$}, and that $\boldsymbol u\in W_0^{1,p}(\Omega)$  solves
     \begin{equation}\label{eq regularity}
        -\textbf{div}(A(x,\boldsymbol u)|D\boldsymbol u|^{p-2}D\boldsymbol u)+\frac1p\nabla_{\boldsymbol s}A(x,\boldsymbol u)|D\boldsymbol u|^p=\boldsymbol g(x,\boldsymbol u)\qquad \text{in $\mathcal D'(\Omega)$.}
    \end{equation}
    Then $\boldsymbol u\in L^\infty(\Omega)$.
\end{theorem}
\begin{proof}
\textbf{Step 1}. For every $s\ge0$, there exists a constant $C(s)>0$  such that 
\[
\left(\int_\Omega|\boldsymbol u|^{\frac{p^*}{p}(2s+p)}\right)^{\frac{p}{p^*}}\le C(s)(1+|\boldsymbol u|_{2s+q}^{2s+q}).
\]

Consider
    \[
    \boldsymbol \varphi=\boldsymbol u\min\{|\boldsymbol u|^{2s},L^2\},\qquad \boldsymbol\psi=\boldsymbol u\left(\min\{|\boldsymbol u|^{2s},L^2\}\right)^{\frac1p}.
    \]
    Testing \eqref{eq regularity} with $\boldsymbol\varphi$, we obtain:
    \[
    \int_\Omega A(x,\boldsymbol u)|D\boldsymbol u|^{p-2}D\boldsymbol u\cdot D\boldsymbol\varphi+\frac1p\int_\Omega\left(\nabla_{\boldsymbol s}A(x,\boldsymbol u)\cdot\boldsymbol\varphi\right)|D\boldsymbol u|^p=\int_\Omega\boldsymbol g(x,\boldsymbol u)\cdot\boldsymbol\varphi.
    \]
    From \eqref{a.3}, we deduce that
    \begin{equation}
    I:=\int_\Omega A(x,\boldsymbol u)|D\boldsymbol u|^{p-2}D\boldsymbol u\cdot D\boldsymbol\varphi\le\int_\Omega \boldsymbol g(x,\boldsymbol u)\cdot\boldsymbol\varphi.
    \end{equation}
    Moreover,
    \begin{align}
        I&=\int_\Omega A(x,\boldsymbol u)|D\boldsymbol u|^{p}\min\{|\boldsymbol u|^{2s},L^2\}+2s\int_{\{|\boldsymbol u|^s\le L\}}A(x,\boldsymbol u)|\boldsymbol u|^{2s-1}|D\boldsymbol u|^{p-2}D\boldsymbol u\cdot(\boldsymbol u\otimes\nabla|\boldsymbol u|)\\
        &=\int_\Omega A(x,\boldsymbol u)|D\boldsymbol u|^{p}\min\{|\boldsymbol u|^{2s},L^2\}+2s\int_{\{|\boldsymbol u|^s\le L\}}A(x,\boldsymbol u)|\boldsymbol u|^{2s}|D\boldsymbol u|^{p-2}|\nabla|\boldsymbol u||^2\\
        &\underbrace{\ge}_{\eqref{a.2}}\int_\Omega A(x,\boldsymbol u)|D\boldsymbol u|^p\min\{|\boldsymbol u|^{2s},L^2\},
    \end{align}
    where we use:
    \[
    D\boldsymbol u\cdot(\boldsymbol u\otimes\nabla|\boldsymbol u|)=\nabla|\boldsymbol u|\cdot\left[(D\boldsymbol u)^{\top}\boldsymbol u\right]=\frac{|(D\boldsymbol u)^{\top}\boldsymbol u|^2}{|\boldsymbol u|}=|\boldsymbol u|\cdot|\nabla|\boldsymbol u||^2.
    \]
    Now, we treat the quantity
    \[
    \int_\Omega A(x,\boldsymbol u)|D\boldsymbol \psi|^{p}.
    \]
    Since
    \[
    D\boldsymbol\psi=
    \begin{cases}
        |\boldsymbol u|^{\frac{2s}{p}}D\boldsymbol u+\frac{2s}{p}|\boldsymbol u|^{\frac{2s}{p}-1}\boldsymbol u\otimes\nabla|\boldsymbol u|&\text{if }|\boldsymbol u|^s\le L\\
        L^{\frac2p}D\boldsymbol u&\text{if }|\boldsymbol u|^s\ge L,
    \end{cases}
    \]
    we have that
    \[
|D\boldsymbol\psi|^p=\left(D\boldsymbol\psi\cdot D\boldsymbol\psi\right)^{\frac{p}{2}}=
\begin{cases}
|\boldsymbol u|^{2s}
\left(
|D\boldsymbol u|^2
+\dfrac{4s}{p}\left(1+\dfrac{s}{p}\right)
|\nabla|\boldsymbol u||^2
\right)^{\frac{p}{2}}
& \text{if } |\boldsymbol u|^s\le L, \\[10pt]
L^{2}\,|D\boldsymbol u|^p
& \text{if } |\boldsymbol u|^s\ge L.
\end{cases}
\]
Hence,
\begin{align}
    \int_\Omega A(x,\boldsymbol u)|D\boldsymbol \psi|^p&\le \int_{\{|\boldsymbol u|^s\le L\}}A(x,\boldsymbol u)|\boldsymbol u|^{2s}
\left(
|D\boldsymbol u|^2
+\dfrac{4s}{p}\left(1+\dfrac{s}{p}\right)
|\nabla|\boldsymbol u||^2
\right)^{\frac{p}{2}}\\
&\quad+\int_{\{|\boldsymbol u|^s\ge L\}}A(x,\boldsymbol u)L^2|D\boldsymbol u|^p\le C_1\int_{\{|\boldsymbol u|^s\le L\}}A(x,\boldsymbol u)|\boldsymbol u|^{2s}|D\boldsymbol u|^p\\
&\quad+\int_{\{|\boldsymbol u|^s\ge L\}}A(x,\boldsymbol u)|D\boldsymbol u|^p\min\{|\boldsymbol u|^{2s},L^2\}\\
&\le C_2\int_\Omega A(x,\boldsymbol u)|D\boldsymbol u|^p\min\{|\boldsymbol u|^{2s},L^2\},
\end{align}
where
\[C_1=C_1(s)=\kappa_p\left[1+\dfrac{4s}{p}\left(1+\dfrac{s}{p}\right)\right]^{\frac p2}\ge 1, \qquad \kappa_p=
\begin{cases}
2^{\frac p2-1} &\text{if } p\ge2\\
 1 &\text{if } 1<p\le2,
 \end{cases}
 \]
 and $C_2=2C_1(s)$.

By \eqref{a.2}, we get
\begin{align}
    \nu_0\int_\Omega|D\boldsymbol\psi|^p\le C_2\int_\Omega A(x,\boldsymbol u)|D\boldsymbol u|^p\min\{|\boldsymbol u|^{2s},L^2\}\le C_2I\le C_2 \int_\Omega \boldsymbol g(x,\boldsymbol u)\cdot\boldsymbol\varphi.
\end{align}
Now, \eqref{g.1} implies that

\begin{align}
    \nu_0\int_\Omega|D\boldsymbol\psi|^p&\le C_2\int_\Omega a(x)|\boldsymbol \varphi|+C_2b\int_\Omega |\boldsymbol u|^{q-1}|\boldsymbol \varphi|\\
    &\le C_2\int_\Omega a(x)|\boldsymbol u|\min\{|\boldsymbol u|^{2s},L^2\}+C_2b\int_\Omega|\boldsymbol u|^{2s+q}\\
    &\le C_2\int_{\{|\boldsymbol u|\le1\}}a(x)|\boldsymbol u|\min\{|\boldsymbol u|^{2s},L^2\}+C_2\int_\Omega a(x)|\boldsymbol u|^{p}\min\{|\boldsymbol u|^{2s},L^2\}\\
    &\quad+C_2b\int_\Omega|\boldsymbol u|^{2s+q}\le C_3+C_2\int_\Omega a(x)|\boldsymbol \psi|^p+C_4|\boldsymbol u|_{2s+q}^{2s+q},
\end{align}
where 
\[
C_3=C_2(s)|a|_1,\ \ \text{and } C_4=C_2(s)b.
\]
For every $K>0$, by H\"older's inequality, we have that
\begin{align}
    \int_\Omega a(x)|\boldsymbol \psi|^p&=\int_{\{a(x)\le K\}}a(x)|\boldsymbol\psi|^p+\int_{\{a(x)\ge K\}}a(x)|\boldsymbol\psi|^p\\
    &\le K|\boldsymbol u|_{2s+p}^{2s+p}+\left(\int_{\{a(x)\ge K\}}|a(x)|^{\frac Np}\right)^{\frac pN}|\boldsymbol \psi|_{p^*}^{p}\\
    &=K|\boldsymbol u|_{2s+p}^{2s+p}+\varepsilon(K)|\boldsymbol \psi|_{p^*}^{p},
\end{align}
where $\varepsilon(K)\to0$ as $K\to+\infty$. 

Thus, by Sobolev's embedding, there exists $S>0$ such that
\begin{align}
   \nu_0S |\boldsymbol\psi|_{p^*}^p\le\nu_0\int_\Omega|D\boldsymbol \psi|^p\le C_3+C_2K|\boldsymbol u|_{2s+p}^{2s+p}+C_2\varepsilon(K)|\boldsymbol \psi|_{p^*}^p+C_4|\boldsymbol u|_{2s+q}^{2s+q}.
\end{align}

Taking $K>0$  such that
\[
\nu_0S-C_2\varepsilon(K)\ge\frac{\nu_0S}{2},
\]
 we obtain 
\[
\frac{\nu_0S}{2}|\boldsymbol\psi|_{p^*}^p\le C_3+C_2K|\boldsymbol u|_{2s+p}^{2s+p}+C_4|\boldsymbol u|_{2s+q}^{2s+q}.
\]
Passing to the limit as $L\to+\infty$, we conclude that 
\[
\left(\int_\Omega|\boldsymbol u|^{\frac{p^*}{p}(2s+p)}\right)^{\frac {p}{p^*}}\le C_5(s)\left(1+|\boldsymbol u|_{2s+p}^{2s+p}+|\boldsymbol u|_{2s+q}^{2s+q}\right),
\]
where
\begin{align}
C_5(s)&=\frac{2}{\nu_0S}\left(C_3(s)+C_2(s)K+C_4(s)\right)\\
&\le\frac{4\kappa_p}{\nu_0 S}
\left[1+\frac{4s}{p}\left(1+\frac{s}{p}\right)\right]^{\frac p2}
\left(
|a|_{1} + K + b
\right).
\end{align}
Exploiting the embedding $L^{2s+q}(\Omega)\hookrightarrow L^{2s+p}(\Omega)$, we deduce that
\begin{align}
\left(\int_\Omega|\boldsymbol u|^{\frac{p^*}{p}(2s+p)}\right)^{\frac {p}{p^*}}&\le C_5(s)\left[1+\left(1+|\Omega|^{\frac{q-p}{2s+q}}\right)|\boldsymbol u|_{2s+q}^{2s+q}\right]\\
&\le C_5(s)\left[1+\bigl(1+\max\{1,|\Omega|\}\bigr)|\boldsymbol u|_{2s+q}^{2s+q}\right],
\end{align}
and the first step follows.

\textbf{Step 2}. A Moser-type iteration.

We define $\sigma:=2s+q$, and the following iteration:
\begin{equation}\label{iteration sigma_n}
\sigma_{n+1}=\frac{p^*}{p}\left(\sigma_n-(q-p)\right),\quad  \sigma_0=q.
\end{equation}
Then
\[
|\boldsymbol u|_{\sigma_{n+1}}\le \left\{C_6\left[1+\frac{4s}{p}\left(1+\frac{s}{p}\right)\right]^{\frac p2}\right\}^{\frac{1}{\sigma_n-(q-p)}}\left(1+|\boldsymbol u|_{\sigma_n}^{\sigma_n}\right)^{\frac{1}{\sigma_n-(q-p)}},
\]
where $C_6>0$ is a positive constant independent of $s=s_n=\frac{\sigma_n-q}{2}$.

Set
\[
\delta_n:=\sigma_n-(q-p)>0,\quad \text{and}\quad M_n:=C_6\left[1+\frac{4s_n}{p}\left(1+\frac{s_n}{p}\right)\right]^{\frac p2}.
\]
Since $1+t^{\sigma_n}\le (1+t)^{\sigma_n}$ for $t\ge0$, we obtain
\[
|\boldsymbol u|_{\sigma_{n+1}}
\le M_n^{\frac1{\delta_n}}(1+|\boldsymbol u|_{\sigma_{n}})^{\alpha_n},
\qquad
\alpha_n:=\frac{\sigma_n}{\delta_n}
=\frac{\sigma_n}{\sigma_n-(q-p)}.
\]
Let $B_n:=1+|\boldsymbol u|_{\sigma_{n}}$. Then
\[
B_{n+1}
=1+|\boldsymbol u|_{\sigma_{n+1}}
\le 1+M_n^{\frac1{\delta_n}}B_n^{\alpha_n}
\le (1+M_n^{\frac1{\delta_n}})B_n^{\alpha_n}\le(1+M_n)^{\frac{1}{\delta_n}}B_n^{\alpha_n}.
\]
Iterating this inequality, one proves by induction that for every $n\ge1$,
\begin{equation}\label{formula induttiva}
B_n
\le
\left(
\prod_{j=0}^{n-1}
(1+M_j)^{\frac1{\delta_j}\prod_{k=j+1}^{n-1}\alpha_k}
\right)
B_0^{\prod_{k=0}^{n-1}\alpha_k}.
\end{equation}
{
Indeed, if 
\[
B_{n-1}
\le
\left(
\prod_{j=0}^{n-2}
(1+M_j)^{\frac1{\delta_j}\prod_{k=j+1}^{n-2}\alpha_k}
\right)
B_0^{\prod_{k=0}^{n-2}\alpha_k}
\]
holds, we have that
\begin{align}
    B_n&\le(1+M_{n-1})^{\frac{1}{\delta_{n-1}}}B_{n-1}^{\alpha_{n-1}}\\
    &\le (1+M_{n-1})^{\frac{1}{\delta_{n-1}}}\left(
\prod_{j=0}^{n-2}
(1+M_j)^{\frac1{\delta_j}\prod_{k=j+1}^{n-2}\alpha_k}
\right)^{\alpha_{n-1}}
B_0^{\alpha_{n-1}\prod_{k=0}^{n-2}\alpha_k}\\
&=(1+M_{n-1})^{\frac{1}{\delta_{n-1}}}\left(
\prod_{j=0}^{n-2}
(1+M_j)^{\frac1{\delta_j}\prod_{k=j+1}^{n-1}\alpha_k}
\right)
B_0^{\prod_{k=0}^{n-1}\alpha_k}\\
&=\left(
\prod_{j=0}^{n-1}
(1+M_j)^{\frac1{\delta_j}\prod_{k=j+1}^{n-1}\alpha_k}
\right)
B_0^{\prod_{k=0}^{n-1}\alpha_k}.
\end{align}
Then, \eqref{formula induttiva} is proved.
}

Since $B_0=1+A_0=1+|\boldsymbol u|_{\sigma_0}$, \eqref{formula induttiva} gives
\[
|\boldsymbol u|_{\sigma_n}
\le
\left(
\prod_{j=0}^{n-1}
(1+M_j)^{\frac1{\delta_j}\prod_{k=j+1}^{n-1}\alpha_k}
\right)
(1+|\boldsymbol u|_{\sigma_0})^{\prod_{k=0}^{n-1}\alpha_k}.
\]

\textbf{Step 3}. Convergence of the iteration and $L^\infty$ bound.

Observe that
\[
\alpha_n
=
\frac{\sigma_n}{\delta_n}
=
1+\frac{q-p}{\delta_n}.
\]
From \eqref{iteration sigma_n}, one has
\begin{align}\label{iteration delta_n}
\delta_n
=
\theta^n\Bigl(p-\frac{q-p}{\theta-1}\Bigr)
+\frac{q-p}{\theta-1},
\qquad
\theta=\frac{p^*}{p}>1.
\end{align}
Since $q<p^*$, the constant
\[
C_*:=p-\frac{q-p}{\theta-1}
\]
is positive and $\delta_n\ge C_*\theta^n$. Hence
\[
\sum_{n=0}^\infty
\log\alpha_n
=
\sum_{n=0}^\infty
\log\!\left(1+\frac{q-p}{\delta_n}\right)
\le
(q-p)\sum_{n=0}^\infty \frac1{\delta_n}
\le
\frac{q-p}{C_*}\sum_{n=0}^\infty \theta^{-n}
<\infty.
\]
Therefore
\begin{equation}\label{convergence alpha_n}
\prod_{n=0}^\infty \alpha_n<\infty.
\end{equation}

Recall that
\[
M_n=C_6\left[1+\frac{4s_n}{p}\left(1+\frac{s_n}{p}\right)\right]^{\frac p2},
\qquad s_n=\frac{\sigma_n-q}{2}.
\]
Since $s_n\ge0$, there exists $C>0$ such that
\[
1+M_n\le C(1+s_n^2)^{\frac p2}.
\]
{
Thus,
\begin{align}
\log(1+M_n)&\le \log C+\frac p2\log(1+s_n^2)\\
&\le C+\frac p2\log(1+s_n)^2\\
&\le C+p\log(1+s_n).
\end{align}
Moreover, from \eqref{iteration delta_n}, one has 
\begin{align}
    \log(1+M_n)&\le C+p\log(1+s_n)\\
    &=C+p\log\left(1+\frac{\sigma_n-q}{2}\right)\\
    &=C+p\log\left(1+\frac{\delta_n-p}{2}\right)\\
    &=C+p\log\left(1+\frac{\theta^n}{2}\Bigl(p-\frac{q-p}{\theta-1}\Bigr)
-\frac12\left(p-\frac{q-p}{\theta-1}\right) \right)\\
&\le C+p\log(1+\alpha_\theta\theta^n),\qquad\alpha_\theta>0.
\end{align}

On the other hand, by \eqref{convergence alpha_n} we have:
\[
\prod_{k=n+1}^\infty \alpha_k \le C_\alpha,
\]
and we also have $\delta_n\ge C_*\theta^n$. Thus, defining
\[
\beta_n:=\frac1{\delta_n}\prod_{k=n+1}^\infty\alpha_k,
\]
we obtain
\[
\beta_n\le \frac{C_\alpha}{\delta_n}\le \frac{C_\alpha}{C_*}\,\theta^{-n}.
\]
Hence
\[
\sum_{n=0}^\infty \beta_n\log(1+M_n)
\le
\frac{C_\alpha}{C_*}(C+p)\sum_{n=0}^\infty \frac{1+\log(1+\alpha_\theta \theta^n)}{\theta^n}<\infty,
\]
which implies that
\[
\prod_{n=0}^\infty (1+M_n)^{\beta_n}<\infty.
\]

\medskip
Since $\sigma_n\to\infty$ and $|\boldsymbol u|_{\sigma_n}$ is uniformly bounded, we conclude that
\[
|\boldsymbol u|_{\infty}=\lim_{n\to\infty}|\boldsymbol u|_{\sigma_n}\le\lim_{n\to+\infty}\left(
\prod_{j=0}^{n-1}
(1+M_j)^{\frac1{\delta_j}\prod_{k=j+1}^{n-1}\alpha_k}
\right)
(1+|\boldsymbol u|_{\sigma_0})^{\prod_{k=0}^{n-1}\alpha_k}<\infty.
\]
}

\end{proof}

\begin{corollary}\label{regularity pLaplacian}
 Assume that $p>1$ and \eqref{g.1} holds with $r\ge\frac Np$.
    Let $\boldsymbol u\in W_0^{1,p}(\Omega)$ be a weak solution of
    \[
    \begin{cases}
    -\boldsymbol\Delta_p\boldsymbol u=\boldsymbol g(x,\boldsymbol u)&\text{in $\Omega$},\\
    \boldsymbol u=0&\text{on $\partial\Omega$.}
    \end{cases}
    \]
    Then $\boldsymbol u\in L^\infty(\Omega)$.
\end{corollary}

\begin{lemma}\label{lemma convergence gradient}
Assume that \eqref{a.1}-\eqref{a.2} hold.
    Let $\{\boldsymbol u_n\}\subset W_0^{1,p}(\Omega)$ and $\{\boldsymbol\omega_n\}$ be two sequences such that
    \[
    \boldsymbol\omega_n:=-\boldsymbol{div}(A(x,\boldsymbol u_n)|D\boldsymbol u_n|^{p-2}D\boldsymbol u_n)+\frac1p\nabla_{\boldsymbol s}A(x,\boldsymbol u_n)|D\boldsymbol u_n|^p.
    \]
    If $\boldsymbol u_n\wto\boldsymbol u$ in $W_0^{1,p}(\Omega)$ and $\boldsymbol\omega_n\to\boldsymbol\omega$ in $W^{-1,p'}(\Omega)$, then
    \[
    |D\boldsymbol u_n|^{p-1}\to |D\boldsymbol u|^{p-1}\quad \text{in $L^1(\Omega)$},\qquad D\boldsymbol u_n\to D\boldsymbol u\quad\text{a.e. in $\Omega$.}
    \]
\end{lemma}
\begin{proof}
    Since $\boldsymbol u_n\wto\boldsymbol u$ in $W_0^{1,p}(\Omega)$ and $\boldsymbol\omega_n\to\boldsymbol\omega$ in $W^{-1,p'}(\Omega)$, we have
\[
\widetilde{\boldsymbol\omega}_n:=\boldsymbol\omega_n+\boldsymbol{div}(A(x,\boldsymbol u_n)|D\boldsymbol u_n|^{p-2}D\boldsymbol u_n)\to\widetilde{\boldsymbol\omega}\quad\text{in $W^{-1,p'}(\Omega)$},
    \]
    and also
    \[
    -\boldsymbol{div}(A(x,\boldsymbol u_n)|D\boldsymbol u_n|^{p-2}D\boldsymbol u_n)=\widetilde{\boldsymbol\omega}_n\qquad\text{in $\mathcal D'(\Omega)$.}
    \]
    Thus, by \cite[Theorem 1]{GradientConvergence98}, we obtain
    \begin{align}\label{convergence gradient}
    |&D\boldsymbol u_n|^{\tau}\to |D\boldsymbol u|^{\tau}\quad \text{in $L^1(\Omega)$, for every $1\le\tau<p$,}\\
    &D\boldsymbol u_n\to D\boldsymbol u\quad\text{a.e. in $\Omega$.}
    \end{align}
    If $p\ge2$, the claim follows by taking $\tau=p-1$ in \eqref{convergence gradient}. 
    
    Otherwise, $p-1,2-p\in(0,1)$, and by H\"older's inequality with exponents $\frac{1}{p-1},\frac{1}{2-p}$:
    \begin{align}
        \int_\Omega|D\boldsymbol u_n-D\boldsymbol u|^{p-1}\le\left(\int_\Omega|D\boldsymbol u_n-D\boldsymbol u|\right)^{p-1}|\Omega|^{2-p}.
    \end{align}
    Now, the right-hand side goes to zero from \eqref{convergence gradient} with $\tau=1$.
\end{proof}

\begin{proposition}\label{chiusura debole di soluzioni} 
Assume that \eqref{a.1}-\eqref{a.2}, \eqref{psi ipotesi} hold.
    Let $\{\boldsymbol u_n\}\subset W_0^{1,p}(\Omega)$ be a weakly convergent sequence such that the sequence $\{\boldsymbol\omega_n\}$, given by
    \[
    \boldsymbol\omega_n:=-\boldsymbol{div}(A(x,\boldsymbol u_n)|D\boldsymbol u_n|^{p-2}D\boldsymbol u_n)+\frac1p\nabla_{\boldsymbol s}A(x,\boldsymbol u_n)|D\boldsymbol u_n|^p,
    \]
    is strongly convergent in $W^{-1,p'}(\Omega)$. Then 
    \[
    -\boldsymbol{div}(A(x,\boldsymbol u)|D\boldsymbol u|^{p-2}D\boldsymbol u)+\frac1p\nabla_{\boldsymbol s}A(x,\boldsymbol u)|D\boldsymbol u|^p=\boldsymbol \omega\quad\text{in $\mathcal D'(\Omega)$},
    \]
    where $\boldsymbol u$ is the weak limit of $\boldsymbol u_n$ and $\boldsymbol \omega$ is the strongly limit of $\boldsymbol\omega_n$ in $W_0^{1,p}(\Omega)$ and $W^{-1,p'}(\Omega)$ respectively.
\end{proposition}
\begin{proof}

Let $\psi:\R\to\R$ be a Lipschitz continuous function as in \eqref{psi ipotesi}, and  consider the test function
\[
\boldsymbol v_n=\varphi H(|\boldsymbol u-\boldsymbol u_n|)\left(
\sigma_1e^{\sigma_1\psi(u_1)-\sigma_1\psi(u_{n,1})},\dots,\sigma_me^{\sigma_m\psi(u_m)-\sigma_m\psi(u_{n,m})}\right),
\]
where $\sigma_1,\dots,\sigma_m\in\{-1,1\},\varphi\in C_c^\infty(\Omega;\R),\varphi\ge0$,
and $H\in C^1(\R;\R)$ is a cut-off function such that
$H(s)=1$ for every $-\frac12\le s\le\frac12$, $H(s)=0$ for every $s\notin(-1,1)$.

Notice that $\boldsymbol v_n\in W_0^{1,p}(\Omega)\cap L^\infty(\Omega)$, and $\boldsymbol v_n\wto (\sigma_1\varphi,\dots,\sigma_m\varphi)$ in $W_0^{1,p}(\Omega)$. 

By Theorem \ref{brezis browder}, we obtain
\begin{align*}
\langle\boldsymbol\omega_n,\boldsymbol v_n\rangle&=\int_\Omega A(x,\boldsymbol u_n)|D\boldsymbol u_n|^{p-2}\sum_{i=1}^m\left[\nabla u_{n,i}\cdot\nabla\left(\varphi_n\right)\sigma_ie^{\sigma_i\psi(u_i)-\sigma_i\psi(u_{n,i})}\right]\\
&\quad+\int_\Omega A(x,\boldsymbol u_n)|D\boldsymbol u_n|^{p-2}\sum_{i=1}^m\left[\nabla u_{n,i}\cdot\nabla(\psi(u_i)-\psi(u_{n,i}))e^{\sigma_i\psi(u_i)-\sigma_i\psi(u_{n,i})}
\right]\varphi_n\\
&\quad+\int_\Omega\sum_{i=1}^m\frac{\sigma_i}{p}\partial_{s_i}A(x,\boldsymbol u_n)e^{M\sigma_i(u_i-u_{n,i})}\varphi_n|D\boldsymbol u_n|^p,
\end{align*}
where $\varphi_n=\varphi H(|\boldsymbol u-\boldsymbol u_n|)$.

 Consider 
\[
f_n(x):= A(x,\boldsymbol u_n)|D\boldsymbol u_n|^{p-2}\sum_{i=1}^m\left[\sigma_i\nabla u_{n,i}\cdot\nabla\varphi_n+\varphi_n\nabla u_{n,i}\cdot\nabla\psi( u_i) \right]e^{\sigma_i\psi(u_i)-\sigma_i\psi(u_{n,i})}.
\]
According to \eqref{a.1}, we have the following bound:
\begin{equation}\label{bound f_n}
|f_n(x)|\le C_1|D\boldsymbol u_n|^{p-1}\cdot|\nabla\varphi_n|+C_1|D\boldsymbol u_n|^{p-1}\cdot|D\boldsymbol u|,\qquad\text{for some $C_1>0$}.
\end{equation}
Moreover, up to a subsequence,
\begin{align}
    |D\boldsymbol u_n|^{p-1}\to |D\boldsymbol u|^{p-1}\quad\text{in $L^1(\Omega)$},\qquad D\boldsymbol u_n\to D\boldsymbol u\quad\text{a.e. in $\Omega$},
\end{align}
 by  Lemma \ref{lemma convergence gradient}. Thus, by Lebesgue's Theorem, $f_n\to f$ in $L^1(\Omega)$ where
\[
f(x):=A(x,\boldsymbol{u})|D\boldsymbol{u}|^{p-2}\left[D\boldsymbol u\cdot D(\boldsymbol\sigma\varphi)+\sum_{i=1}^m\psi'(u_i)|\nabla u_i|^2\varphi\right],\quad \boldsymbol\sigma\varphi:=(\sigma_1\varphi,\dots,\sigma_m\varphi).
\]

Now, we treat the integral
\[
\int_\Omega\sum_{i=1}^m\left[\frac{\sigma_i}{p}\partial_{s_i}A(x,\boldsymbol u_n)|D\boldsymbol u_n|^p-A(x,u_n)|D\boldsymbol u_n|^{p-2}\nabla u_{n,i}\cdot\nabla\psi(u_{n,i})\right] e^{\sigma_i\psi(u_i)-\sigma_i\psi(u_{n,i})}\varphi_n.
\]
From \eqref{psi ipotesi}, we have that
\begin{align}
    h_n:&=\sum_{i=1}^m\left[\frac{\sigma_i}{p}\partial_{s_i}A(x,\boldsymbol u_n)|D\boldsymbol u_n|^2-A(x,u_n)|\nabla u_{n,i}|^2\psi'(u_{n,i})\right]|D\boldsymbol u_n|^{p-2} e^{\sigma_i\psi(u_i)-\sigma_i\psi(u_{n,i})}\varphi_n\\
    &=\bigg[\sum_{i=1}^m\left(\frac{\sigma_i}{p}\partial_{s_i}A(x,\boldsymbol u_n)e^{\sigma_i\psi(u_i)-\sigma_i\psi(u_{n,i})}\right)\sum_{j=1}^m|\nabla u_{n,j}|^2\\
    &\quad-\sum_{i=1}^mA(x,\boldsymbol u_n)\psi'(u_{n,i})e^{\sigma_i\psi(u_i)-\sigma_i\psi(u_{n,i})}|\nabla u_{n,i}|^2\bigg]|D\boldsymbol u_n|^{p-2}\varphi_n\\
    &=\sum_{j=1}^m\bigg[\sum_{i=1}^m\left(\frac{\sigma_i}{p}\partial_{s_i}A(x,\boldsymbol u_n)e^{\sigma_i\psi(u_i)-\sigma_i\psi(u_{n,i})}\right)\\
    &\quad-A(x,\boldsymbol u_n)\psi'(u_{n,j})e^{\sigma_j\psi(u_j)-\sigma_j\psi(u_{n,j})}\bigg]|\nabla u_{n,j}|^2|D\boldsymbol u_n|^{p-2}\varphi_n\le0.
\end{align}
{
Since $D\boldsymbol{u}_n\to D\boldsymbol u$ a.e. in $\Omega$ by Lemma \ref{lemma convergence gradient}, by Fatou's Lemma we deduce that
\[
\limsup_{n\to+\infty}\int_\Omega h_n(x)\le\frac1p\int_\Omega\nabla_{\boldsymbol s}A(x,\boldsymbol u)\cdot(\boldsymbol\sigma\varphi)|D\boldsymbol u|^p-\sum_{i=1}^m\int_\Omega A(x,\boldsymbol u)|D\boldsymbol u|^{p-2}|\nabla u_i|^2\psi'(u_i)\varphi. 
\]
}
In particular,
\begin{align*}
\langle\boldsymbol\omega,\boldsymbol\sigma\varphi\rangle=\limsup_{n\to+\infty}\langle\boldsymbol\omega_n,\boldsymbol v_n\rangle\le\int_\Omega A(x,\boldsymbol u)|D\boldsymbol u|^{p-2}D\boldsymbol u\cdot D(\boldsymbol\sigma\varphi)+\frac1p\int_\Omega\nabla_{\boldsymbol s}A(x,\boldsymbol u)\cdot(\boldsymbol\sigma \varphi)|D\boldsymbol u|^p,
\end{align*}
for every $\boldsymbol\sigma=(\sigma_1,\dots,\sigma_m)$ with $\sigma_1,\dots,\sigma_m\in\{-1,1\}$ and for every $\varphi\in C_c^\infty(\Omega;\R)$ with $\varphi\ge0$. Since
\[
(\sigma_1(-\varphi),\dots,\sigma_m(-\varphi))=(-\sigma_1\varphi,\dots,-\sigma_m\varphi),\qquad \text{with $-\sigma_1,\dots,-\sigma_m\in\{-1,1\}$,}
\]
we conclude that
\[
\langle\boldsymbol\omega,\boldsymbol\sigma\varphi\rangle=\int_\Omega A(x,\boldsymbol u)|D\boldsymbol u|^{p-2}D\boldsymbol u\cdot D(\boldsymbol\sigma\varphi)+\frac1p\int_\Omega\nabla_{\boldsymbol s}A(x,\boldsymbol u)\cdot(\boldsymbol\sigma \varphi)|D\boldsymbol u|^p,
\]
for every $\boldsymbol\sigma=(\sigma_1,\dots,\sigma_m)$ with $\sigma_1,\dots,\sigma_m\in\{-1,1\}$ and for every $\varphi\in C_c^\infty(\Omega;\R)$ with $\varphi\ge0$. 

Since the vectors $\boldsymbol{\sigma}\in\{-1,1\}^m$ generate $\mathbb{R}^m$, 
every $\boldsymbol{\varphi}\in C_c^\infty(\Omega;\mathbb{R}^m)$ 
can be written as a finite linear combination of functions of the form 
$\boldsymbol{\sigma}\,\varphi$, with $\varphi\in C_c^\infty(\Omega;\R)$, $\varphi\ge 0$.
{
Indeed, we have
\begin{align}
    \boldsymbol\varphi&=\sum_{i=1}^m\varphi_i\boldsymbol e_i=\sum_{i=1}^m\left(\varphi_i+|\varphi_i|_\infty \eta_i-|\varphi_i|_\infty \eta_i\right)\boldsymbol e_i\\
    &=\sum_{i=1}^m\frac{1}{2^m}\sum_{\boldsymbol\sigma\in\{-1,1\}^m}\left(\varphi_i+|\varphi_i|_\infty \eta_i-|\varphi_i|_\infty \eta_i\right)\sigma_i\boldsymbol\sigma\\
    &=\frac{1}{2^m}\sum_{\boldsymbol\sigma\in\{-1,1\}^m}\sum_{i=1}^m\bigl[\left(\varphi_i+|\varphi_i|_\infty \eta_i\bigr)\sigma_i\boldsymbol\sigma-|\varphi_i|_{\infty}\eta_i\sigma_i\boldsymbol\sigma\right]\\
    &=\frac{1}{2^m}\sum_{\boldsymbol\sigma\in\{-1,1\}^m}\sum_{i=1}^m\bigl[(\sigma_i\boldsymbol\sigma)\left(\varphi_i+|\varphi_i|_\infty \eta_i\right)+(-\sigma_i\boldsymbol\sigma)|\varphi_i|_{\infty}\eta_i\bigr],
\end{align}
where $\{\boldsymbol e_i\}_{i=1}^m$ is the canonical basis of $\R^m$ and $\eta_i\in C_c^\infty(\Omega;[0,1])$ is a cut-off function such that $\eta_i\equiv1$ in $\text{supp}(\varphi_i)$ for every $i=1,\dots,m$.
}

Hence, it follows that
\[
\langle\boldsymbol\omega,\boldsymbol\varphi\rangle=\int_\Omega A(x,\boldsymbol u)|D\boldsymbol u|^{p-2}D\boldsymbol u\cdot D\boldsymbol\varphi+\frac1p\int_\Omega\left(\nabla_{\boldsymbol s}A(x,\boldsymbol u)\cdot\boldsymbol \varphi\right)|D\boldsymbol u|^p,\quad \text{for every $\boldsymbol\varphi\in C_c^\infty(\Omega)$}.
\]
\end{proof}

\begin{theorem}\label{compactness}
Assume that \eqref{a.1}-\eqref{a.3} and \eqref{psi ipotesi} hold.
     Let $\{\boldsymbol u_n\}\subset W_0^{1,p}(\Omega)$ be a bounded sequence such that
    \[
    \boldsymbol\omega_n:=-\boldsymbol{div}(A(x,\boldsymbol u_n)|D\boldsymbol u_n|^{p-2}D\boldsymbol u_n)+\frac1p\nabla_{\boldsymbol s}A(x,\boldsymbol u_n)|D\boldsymbol u_n|^p\in\mathcal D'(\Omega)
    \]
    is strongly convergent in $W^{-1,p'}(\Omega)$. Then $\{\boldsymbol u_n\}$ admits a convergent subsequence in $W_0^{1,p}(\Omega)$.    \end{theorem}

\begin{proof}
Since $\{\boldsymbol u_n\}$ is bounded in $W_0^{1,p}(\Omega)$ and
$\{\boldsymbol \omega_n\}$ converges strongly in $W^{-1,p'}(\Omega)$,
up to a subsequence, we have
\begin{align}\label{convergenza sottosuccessione}
\boldsymbol u_n \rightharpoonup \boldsymbol u 
\quad \text{in } W_0^{1,p}(\Omega),
\qquad  \boldsymbol u_n\to \boldsymbol u\quad\text{a.e. in $\Omega$}\qquad
\boldsymbol \omega_n \to \boldsymbol \omega 
\quad \text{in } W^{-1,p'}(\Omega).
\end{align}
Moreover, by Theorem \ref{brezis browder} and Proposition \ref{chiusura debole di soluzioni}, we also obtain
\[
\int_\Omega A(x,\boldsymbol u_n)|D\boldsymbol u_n|^{p}+\frac1p\int_\Omega\left( \nabla_{\boldsymbol s}A(x,\boldsymbol u_n)\cdot\boldsymbol u_n\right)|D\boldsymbol u_n|^p=\langle\boldsymbol \omega_n,\boldsymbol u_n\rangle,
\]
\[
\int_\Omega A(x,\boldsymbol u)|D\boldsymbol u|^{p}+\frac1p\int_\Omega\left( \nabla_{\boldsymbol s}A(x,\boldsymbol u)\cdot\boldsymbol u\right)|D\boldsymbol u|^p=\langle\boldsymbol \omega,\boldsymbol u\rangle.
\]
Now, let us prove that
\[
\limsup_{n\to+\infty}\int_\Omega A(x,\boldsymbol u_n)|D\boldsymbol u_n|^p\le \int_\Omega A(x,\boldsymbol u)|D\boldsymbol u|^p.
\]

According to Lemma \ref{lemma convergence gradient}, $|D\boldsymbol u_n|^p\to |D\boldsymbol u|^p$ a.e. in $\Omega$, and by \eqref{a.3}, we have
\begin{align}
    \nabla_{\boldsymbol s}A(x,\boldsymbol u_n)\cdot\boldsymbol u_n\ge0\qquad\text{for a.e. $x\in\Omega$.}
\end{align}
Hence,
\[
\int_\Omega \left(\nabla_{\boldsymbol s}A(x,\boldsymbol u)\cdot \boldsymbol u\right)|D\boldsymbol u|^p\le\liminf_{n\to+\infty}\int_\Omega \left(\nabla_{\boldsymbol s}A(x,\boldsymbol u_n)\cdot\boldsymbol u_n\right)|D\boldsymbol u_n|^p,
\]
by Fatou's Lemma.
Thus,
\begin{align*}
\limsup_{n\to+\infty}\int_\Omega A(x,\boldsymbol u_n)|D\boldsymbol u_n|^p&=\limsup_{n\to+\infty}\left\{\langle\boldsymbol \omega_n,\boldsymbol u_n\rangle-\frac1p\int_\Omega \left(\nabla_{\boldsymbol s}A(x,\boldsymbol u_n)\cdot\boldsymbol u_n\right)|D\boldsymbol u_n|^p
\right\}\\
&=\langle\boldsymbol\omega,\boldsymbol u\rangle-\liminf_{n\to+\infty}\int_\Omega \left(\nabla_{\boldsymbol s}A(x,\boldsymbol u_n)\cdot\boldsymbol u_n\right)|D\boldsymbol u_n|^p\\
&\le\langle\boldsymbol\omega,\boldsymbol u\rangle-\int_\Omega \left(\nabla_{\boldsymbol s}A(x,\boldsymbol u)\cdot \boldsymbol u\right)|D\boldsymbol u|^p\\
&=\int_\Omega A(x,\boldsymbol u)|D\boldsymbol u|^p.
\end{align*}
In particular, the claim is proved, and we have that
\begin{align}\label{limite ADu_n}
\lim_{n\to+\infty}\int_\Omega A(x,\boldsymbol u_n)|D\boldsymbol u_n|^p=\int_\Omega A(x,\boldsymbol u)|D\boldsymbol u|^p.
\end{align}
Finally, we can show that $\boldsymbol u_n\to\boldsymbol u$ in $W_0^{1,p}(\Omega)$. Notice that
\begin{align*}
    |D\boldsymbol u_n|^p-|D\boldsymbol u|^p&=\frac{A(x,\boldsymbol u_n)|D\boldsymbol u_n|^p+A(x,\boldsymbol u)|D\boldsymbol u|^p-A(x,\boldsymbol u)|D\boldsymbol u|^p-A(x,\boldsymbol u_n)|D\boldsymbol u|^p}{A(x,\boldsymbol u_n)}\\
    &=\frac{A(x,\boldsymbol u)-A(x,\boldsymbol u_n)}{A(x,\boldsymbol u_n)}|D\boldsymbol u|^p+\frac{A(x,\boldsymbol u_n)|D\boldsymbol u_n|^p-A(x,\boldsymbol u)|D\boldsymbol u|^p}{A(x,\boldsymbol u_n)}
\end{align*}

Therefore, by \eqref{convergenza sottosuccessione}-\eqref{limite ADu_n} and \eqref{a.1}-\eqref{a.2}, we deduce that

\begin{align*}
    \left|\int_\Omega|D\boldsymbol u_n|^p-\int_\Omega |D\boldsymbol u|^p\right|&\le\left|\int_\Omega\frac{A(x,\boldsymbol u)-A(x,\boldsymbol u_n)}{A(x,\boldsymbol u_n)}|D\boldsymbol u|^p\right|\\
    &\quad+\left|\int_\Omega\frac{A(x,\boldsymbol u_n)|D\boldsymbol u_n|^p-A(x,\boldsymbol u)|D\boldsymbol u|^p}{A(x,\boldsymbol u_n)}\right|\\
    &\le\frac{1}{\nu_0}\int_\Omega\bigl|A(x,\boldsymbol u)-A(x,\boldsymbol u_n)\bigr|\cdot|D\boldsymbol u|^p\\
    &\quad+\frac{1}{\nu_0}\int_\Omega\bigl|A(x,\boldsymbol u_n)|D\boldsymbol u_n|^p-A(x,\boldsymbol u)|D\boldsymbol u|^p\bigr|\longrightarrow0\qquad\text{as $n\to+\infty$,}
\end{align*}
and $\boldsymbol u_n\to\boldsymbol u$ in $W_0^{1,p}(\Omega)$.
\end{proof}

\begin{corollary}\label{CPS condition corollary}
Assume that \eqref{a.1}-\eqref{a.3}, \eqref{psi ipotesi} and \eqref{g.1} hold. 
   For every $c\in\R$ the following facts are equivalent:
    \begin{itemize}
        \item[$(a)$] $\mathcal J$ satisfies the $(CPS)_c$-condition;
        \item[$(b)$] every $(CPS)_c$-sequence for $\mathcal J$ is bounded in $W_0^{1,p}(\Omega)$.
    \end{itemize}
\end{corollary}

\begin{proof}
   The implication $(a)\implies(b)$ is a consequence of the strong convergence.
   
   Conversely, let $\{\boldsymbol u_n\}\subset W_0^{1,{p}}(\Omega)$ be a $(CPS)_c$-sequence, and assume that $\boldsymbol u_n\wto \boldsymbol u$ in $W_0^{1,{p}}(\Omega)$ up to a subsequence. 
   
   Thus, up to a further subsequence, $\boldsymbol u_{n}\to \boldsymbol u$ in $L^{t}(\Omega)$ with $t\in[p,p^*)$, and we have also that
    \[\int_\Omega \boldsymbol g(x,\boldsymbol u_n)\cdot\boldsymbol v\to\int_\Omega \boldsymbol g(x,\boldsymbol u)\cdot\boldsymbol v,\ \ \text{for every $\boldsymbol v\in W_0^{1,p}(\Omega)$ }.
    \]
    According to Theorem \ref{compactness} with $\boldsymbol\omega_n=\boldsymbol g(x,\boldsymbol u_n)$, $\{\boldsymbol u_n\}$ admits a strongly convergent subsequence in $W_0^{1,p}(\Omega)$.
\end{proof}
\section{Multiplicity result}
\begin{proposition} \label{stima per limitatezza}
Assume that \eqref{a.1}-\eqref{a.3} and \eqref{g.1} hold.
Let $\{\boldsymbol u_n\}\subset W_0^{1,p}(\Omega)$ be a $(CPS)$-sequence for $\mathcal J$ at level $c\in\R$. Then, for every $R>0$ and $\varepsilon>0$, there exists a constant $K=K(R,\varepsilon)>0$ such that
\[
\int_{\{|\boldsymbol u_n|\le R\}}A(x,\boldsymbol u_n)|D\boldsymbol u_n|^p\le\varepsilon\int_{\{|\boldsymbol u_n|\ge R\}}A(x,\boldsymbol u_n)|D\boldsymbol u_n|^p+K,\qquad\text{for every $n\in\N$.}
\]
\end{proposition}
\begin{proof}
Let
\[
\boldsymbol\omega_n:= -\textbf{div}(A(x,\boldsymbol u_n)|D\boldsymbol u_n|^{p-2}D\boldsymbol u_n)+\frac1p\nabla_{\boldsymbol s}A(x,\boldsymbol u_n)|D\boldsymbol u_n|^p-\boldsymbol g(x,\boldsymbol u_n),
\]
take $\delta\in(0,1)$, and consider
\[
\boldsymbol\theta_\delta=
\begin{cases}
\boldsymbol u_n&\text{if }|\boldsymbol u_n|\le R,\\
-\delta\boldsymbol u_n+(R+\delta R)\frac{\boldsymbol u_n}{|\boldsymbol u_n|}&\text{if } R\le|\boldsymbol u_n|\le\frac{R+\delta R}{\delta},\\
\boldsymbol 0& \text{if } |\boldsymbol u_n|\ge\frac{R+\delta R}{\delta}.
\end{cases}
\]
Then, we have $\boldsymbol\theta_\delta\in W_0^{1,p}(\Omega)\cap L^\infty(\Omega)$ and 
\[
D\boldsymbol\theta_\delta=
\begin{cases}
   D \boldsymbol u_n& \text{if }|\boldsymbol u_n|\le R,\\
-\delta D\boldsymbol u_n+\frac{R+\delta R}{|\boldsymbol u_n|}D\boldsymbol u_n-(R+\delta R)\boldsymbol u_n\otimes\nabla\left(\frac{1}{|\boldsymbol u_n|}\right)& \text{if } R\le|\boldsymbol u_n|\le\frac{R+\delta R}{\delta},\\
\boldsymbol 0& \text{if }|\boldsymbol u_n|\ge\frac{R+\delta R}{\delta}.
\end{cases}
\]
In particular,
\begin{equation}\label{stima gradiente}
    |D\boldsymbol \theta_\delta|\le(2+3\delta)|D\boldsymbol u_n|.
\end{equation}
Since $\boldsymbol u_n$ solves
\[
-\textbf{div}(A(x,\boldsymbol u_n)|D\boldsymbol u_n|^{p-2}D\boldsymbol u_n)+\frac1p\nabla_{\boldsymbol s}A(x,\boldsymbol u_n)|D\boldsymbol u_n|^p=\boldsymbol g(x,\boldsymbol u_n)+\boldsymbol\omega_n\quad\text{in $\mathcal D'(\Omega)$},
\]
according to Theorem \ref{brezis browder}, we get
\[
\int_\Omega A(x,\boldsymbol u_n)|D\boldsymbol u_n|^{p-2}D\boldsymbol u_n\cdot D\boldsymbol \theta_\delta+\frac1p\int_\Omega\left(\nabla_{\boldsymbol s}A(x,\boldsymbol u_n)\cdot\boldsymbol\theta_\delta\right)|D\boldsymbol u_n|^p=\int_\Omega\boldsymbol g(x,\boldsymbol u_n)\cdot\boldsymbol\theta_\delta+\langle\boldsymbol \omega_n,\boldsymbol\theta_\delta\rangle.
\]
By \eqref{a.3}, we also deduce that
\[
\frac1p\int_\Omega\left(\nabla_{\boldsymbol s}A(x,\boldsymbol u_n)\cdot\boldsymbol\theta_\delta\right)|D\boldsymbol u_n|^p\ge0.
\]

Now, recall that for every $A\in\mathbb{R}^{m\times N}$,
$\boldsymbol v\in\mathbb{R}^m$, and $\boldsymbol w\in\mathbb{R}^N$,
\[
A\cdot (\boldsymbol v\otimes\boldsymbol w)= \boldsymbol w\cdot (A^{\top}\boldsymbol v).
\]
If $R\le|\boldsymbol u_n|\le\frac{R+\delta R}{\delta}$, a direct calculation yields:
\begin{align*}
D\boldsymbol u_n\cdot D\boldsymbol\theta_\delta&=-\delta |D\boldsymbol u_n|^2+(R+\delta R)\left\{\frac{|D\boldsymbol u_n|^2}{|\boldsymbol u_n|}+D\boldsymbol u_n\cdot\left[\boldsymbol u_n\otimes\nabla\left(\frac{1}{|\boldsymbol u_n|}\right)\right]\right\}\\
&= -\delta |D\boldsymbol u_n|^2+(R+\delta R)\frac{|D\boldsymbol u_n|^2}{|\boldsymbol u_n|}-\frac{R+\delta R}{|\boldsymbol u_n|^3}D\boldsymbol u_n\cdot\left[\boldsymbol u_n\otimes (D\boldsymbol u_n)^{\top}\boldsymbol u_n\right] \\
&=-\delta |D\boldsymbol u_n|^2+(R+\delta R)\frac{|D\boldsymbol u_n|^2}{|\boldsymbol u_n|}-(R+\delta R)\left[\frac{|(D\boldsymbol u_n)^{\top}\boldsymbol u_n|^2}{|\boldsymbol u_n|^3}\right]\\
&\ge-\delta|D\boldsymbol u_n|^2.
\end{align*}

Therefore, we obtain

\begin{align*}
\int_\Omega A(x,\boldsymbol u_n)|D\boldsymbol u_n|^{p-2}D\boldsymbol u_n\cdot D\boldsymbol \theta_\delta\ge\int_{\{|\boldsymbol u_n|\le R\}}A(x,\boldsymbol u_n)|D\boldsymbol u_n|^p-\delta\int_{\{|\boldsymbol u_n|\ge R\}}A(x,\boldsymbol u_n)|D\boldsymbol u_n|^p.
\end{align*}
Thus,
\begin{equation}\label{stima theta_delta}
\begin{aligned}
\int_{\{|\boldsymbol u_n|\le R\}}A(x,\boldsymbol u_n)|D\boldsymbol u_n|^p&-\delta\int_{\{|\boldsymbol u_n|\ge R\}}A(x,\boldsymbol u_n)|D\boldsymbol u_n|^p\\
&\qquad\le \int_\Omega\boldsymbol g(x,\boldsymbol u_n)\cdot\boldsymbol\theta_\delta+\langle\boldsymbol \omega_n,\boldsymbol\theta_\delta\rangle.
\end{aligned}
\end{equation}
Now, from (weighted) Young's inequality and by the strong convergence of $\{\boldsymbol\omega_n\}$ in $W^{-1,p'}(\Omega)$, we have that there exists a constant $C_1=C_1(\delta)>0$ such that
\begin{align*}
\langle\boldsymbol \omega_n,\boldsymbol\theta_\delta\rangle&\le C_1+\delta\|\boldsymbol\theta_\delta\|_{W_0^{1,p}}^p=C_1+\delta\int_\Omega\left|D\boldsymbol\theta_\delta\right|^p.
\end{align*}
By \eqref{stima gradiente}, we have
\begin{equation}\label{stima omega_ntheta_delta}
\begin{aligned}
    \langle\boldsymbol \omega_n,\boldsymbol\theta_\delta\rangle&\le C_1+\delta\int_\Omega\left|D\boldsymbol\theta_\delta\right|^p\le C_1+\delta(2+3\delta) \int_\Omega |D\boldsymbol u_n|^p\\
    &\le C_1+5\delta\int_{\{|\boldsymbol u_n|\le R\}}|D\boldsymbol u_n|^p+5\delta\int_{\{|\boldsymbol u_n|\ge R\}}|D\boldsymbol u_n|^p.
\end{aligned}\end{equation}

Taking into account \eqref{a.2}, we obtain
\[
\langle\boldsymbol \omega_n,\boldsymbol\theta_\delta\rangle\le C_1+\frac{5\delta}{\nu_0}\int_{\{|\boldsymbol u_n|\le R\}}A(x,\boldsymbol u_n)|D\boldsymbol u_n|^p+\frac{5\delta}{\nu_0}\int_{\{|\boldsymbol u_n|\ge R\}}A(x,\boldsymbol u_n)|D\boldsymbol u_n|^p.
\]
Furthermore, \eqref{g.1} implies that there exists a constant $C_2>0$ such that
\begin{align*}
\int_\Omega\boldsymbol g(x,\boldsymbol u_n)\cdot\boldsymbol\theta_\delta\le\int_\Omega a(x)|\boldsymbol\theta_\delta|+\int_\Omega|\boldsymbol u_n|^{q-1}|\boldsymbol\theta_\delta|\le C_2.
\end{align*}
Hence, from \eqref{stima theta_delta}-\eqref{stima omega_ntheta_delta}, we infer that there exists $C_3>0$ such that
\begin{align*}
\left(1-\frac{5\delta}{\nu_0}\right)\int_{\{|\boldsymbol u_n|\le R\}}A(x,\boldsymbol u_n)|D\boldsymbol u_n|^p\le\left(\delta+\frac{5\delta}{\nu_0}\right)\int_{\{|\boldsymbol u_n|\ge R\}}A(x,\boldsymbol u_n)|D\boldsymbol u_n|^p+C_3.
\end{align*}
Choosing $\delta \in (0,1)$ sufficiently small such that
\[
0<\frac{\nu_0\delta+5\delta}{\nu_0-5\delta}\le\varepsilon,
\]
the conclusion follows.
\end{proof}

\begin{theorem}\label{CPS condition}
Assume that \eqref{a.1}-\eqref{a.4} and \eqref{g.1}-\eqref{g.2} hold.
Every $(CPS)$-sequence for $\mathcal J$ is bounded. In particular, if \eqref{psi ipotesi} is also satisfied, the functional $\mathcal J$ satisfies the $(CPS)$-condition.
\end{theorem}
\begin{proof}
Let $\{\boldsymbol u_n\}\subset W_0^{1,p}(\Omega)$ be a $(CPS)$-sequence at level $c\in\R$. We define
\[
\boldsymbol \omega_n= -\textbf{div}(A(x,\boldsymbol u_n)|D\boldsymbol u_n|^{p-2}D\boldsymbol u_n)+\frac1p\nabla_{\boldsymbol s}A(x,\boldsymbol u_n)|D\boldsymbol u_n|^p-\boldsymbol g(x,\boldsymbol u_n).
\]
In particular, $\boldsymbol u_n$ solves
\[
 -\textbf{div}(A(x,\boldsymbol u_n)|D\boldsymbol u_n|^{p-2}D\boldsymbol u_n)+\frac1p\nabla_{\boldsymbol s}A(x,\boldsymbol u_n)|D\boldsymbol u_n|^p=\boldsymbol\omega_n+\boldsymbol g(x,\boldsymbol u_n)\in W^{-1,p'}(\Omega).
\]
Theorem \ref{brezis browder} implies that
\[
\langle\boldsymbol \omega_n,\boldsymbol u_n\rangle=\int_\Omega A(x,\boldsymbol u_n)|D\boldsymbol u_n|^{p}+\frac1p\int_\Omega \left(\nabla_{\boldsymbol s}A(x,\boldsymbol u_n)\cdot\boldsymbol u_n\right)|D\boldsymbol u_n|^p-\int_\Omega \boldsymbol g(x,\boldsymbol u_n)\cdot\boldsymbol u_n.
\]
Moreover, since $\{\boldsymbol u_n\}$ is a $(CPS)$-sequence, there exists $C>0$ such that
\[
C+\|\boldsymbol u_n\|\ge\mathcal J(\boldsymbol u_n)-\frac1\mu\langle\boldsymbol \omega_n,\boldsymbol u_n\rangle.
\]
We compute
\begin{align*}
\mathcal J(\boldsymbol u_n)-\frac1\mu\langle\boldsymbol \omega_n,\boldsymbol u_n\rangle&=\left(\frac1p-\frac1\mu\right)\int_\Omega A(x,\boldsymbol u_n)|D\boldsymbol u_n|^p-\frac{1}{\mu p}\int_\Omega  \left(\nabla_{\boldsymbol s}A(x,\boldsymbol u_n)\cdot\boldsymbol u_n\right)|D\boldsymbol u_n|^p\\
&\quad+\frac1\mu\int_\Omega \boldsymbol g(x,\boldsymbol u_n)\cdot\boldsymbol u_n-\int_\Omega G(x,\boldsymbol u_n).
\end{align*}
According to \eqref{g.2}, we obtain
\[
\frac1\mu\int_\Omega \boldsymbol g(x,\boldsymbol u_n)\cdot\boldsymbol u_n-\int_\Omega G(x,\boldsymbol u_n)\ge -K_1,
\]
where $K_1>0$ is a positive constant.

Furthermore, from \eqref{a.1}-\eqref{a.2} and \eqref{a.4}, we have that
\begin{align*}
\int_\Omega   \left(\nabla_{\boldsymbol s}A(x,\boldsymbol u_n)\cdot\boldsymbol u_n\right)|D\boldsymbol u_n|^p&\le C_0R\int_{\{|\boldsymbol u_n|\le R\}}|D\boldsymbol u_n|^p+\gamma\int_{\{|\boldsymbol u_n|\ge R\} }A(x,\boldsymbol u_n)|D\boldsymbol u_n|^p\\
&\le\frac{C_0R}{\nu_0}\int_{\{|\boldsymbol u_n|\le R\}}A(x,\boldsymbol u_n)|D\boldsymbol u_n|^p +\gamma\int_{\{|\boldsymbol u_n|\ge R\}}A(x,\boldsymbol u_n)|D\boldsymbol u_n|^p.
\end{align*}
Now, according to Proposition \ref{stima per limitatezza}, we deduce that
\[
\int_\Omega   \left(\nabla_{\boldsymbol s}A(x,\boldsymbol u_n)\cdot\boldsymbol u_n\right)|D\boldsymbol u_n|^p\le\left(\frac{C_0R\varepsilon}{\nu_0}+\gamma\right)\int_\Omega A(x,\boldsymbol u_n)|D\boldsymbol u_n|^p+K_2,
\]
for every $\varepsilon>0$ and for some $K_2>0$.
 Thus,
\begin{align*}
C+\|\boldsymbol u_n\|\ge \left[\frac{\mu-p}{\mu p}-\frac{1}{\mu p}\left(\frac{C_0R\varepsilon}{\nu_0}+\gamma\right)\right]\int_\Omega A(x,\boldsymbol u_n)|D\boldsymbol u_n|^p-K_3.
\end{align*}
Since $\gamma<\mu-p$  and \eqref{a.2} holds, we can take 
\[
\varepsilon<(\mu-p-\gamma)\frac{\nu_0}{C_0R}
\]
such that there exists $\nu_1=\nu_1(\nu_0,\varepsilon)>0$ with

\[
C+\|\boldsymbol u_n\|\ge\nu_1\|\boldsymbol u_n\|^p-K_3.
\]
Then, $\{\boldsymbol u_n\}$ is bounded in $W_0^{1,p}(\Omega)$. Finally,  Corollary \ref{CPS condition corollary} concludes the proof.
\end{proof}
\begin{corollary}\label{PS condition}
Assume that \eqref{a.1}-\eqref{a.4}, \eqref{psi ipotesi}, and \eqref{g.1}-\eqref{g.2} hold.
The functional $\mathcal J$ satisfies the $(PS)$-condition.
\end{corollary}
\begin{proof}
Theorem \ref{CPS condition} and Corollary \ref{CPS condition corollary} imply that $\mathcal J$ satisfies the $(CPS)$-condition. Thus, $\mathcal J$ satisfies the $(PS)$-condition by  Proposition \ref{PS e CPS}.
\end{proof}

\begin{proof}[Proof of Theorem \ref{main thm multiplicity}]
According to \eqref{a.1}-\eqref{a.2}, we have that
\[
\mathcal J(\boldsymbol u)\ge\frac{\nu_0}{p}\|\boldsymbol u\|^p-\int_\Omega G(x,\boldsymbol u),\quad \mathcal J(\boldsymbol u)\le\frac{C_0}{p}\|\boldsymbol u\|^p-\int_\Omega G(x,\boldsymbol u).
\]
Thus, the geometric assumption of Theorem \ref{MPequi} can be proved as in \cite{canino2025vectorial}. Moreover, the functional $\mathcal J$ satisfies the $(PS)$-condition by Corollary \ref{PS condition} and the existence of infinitely many weak solutions of \eqref{P} follows.

If \eqref{g.1} holds with {$r\ge\frac Np$}, Theorem \ref{regularity} implies also that any weak solution is bounded, and the proof is complete.
\end{proof}


\begin{thebibliography}{10}

\bibitem{arcoyaboccardo1}
David Arcoya and Lucio Boccardo.
\newblock Critical points for multiple integrals of the calculus of variations.
\newblock {\em Arch. Rational Mech. Anal.}, 134(3):249--274, 1996.


\bibitem{arioli2000existence}
Gianni Arioli and Filippo Gazzola.
\newblock Existence and multiplicity results for quasilinear elliptic
  differential systems: Quasilinear elliptic differential systems.
\newblock {\em Communications in Partial Differential Equations},
  25(1-2):125--153, 2000.




\bibitem{balci2022pointwise}
Anna~Kh. Balci, Andrea Cianchi, Lars Diening, and Vladimir Maz’ya.
\newblock A pointwise differential inequality and second-order regularity for
  nonlinear elliptic systems.
\newblock {\em Mathematische Annalen}, 383(3):1--50, 2022.


\bibitem{brezisbrowder2}
Ha\"im Br\'ezis and Felix~E. Browder.
\newblock Some properties of higher order Sobolev spaces.
\newblock {\em J. Math. Pures Appl.} (9) {\bf 61} (1982).



\bibitem{candela2009some}
Anna~Maria Candela and Giuliana Palmieri.
\newblock Some abstract critical point theorems and applications.
\newblock  {\em Discrete Continuous Dynamical Systems}, Suppl 2009, 133--142, 2009
  
    \bibitem{candela2009infinitely}
Anna~Maria Candela and Giuliana Palmieri.
\newblock Infinitely many solutions of some nonlinear variational equations.
\newblock {\em Calculus of Variations and Partial Differential Equations},
  34(4):495--530, 2009.


\bibitem{caninoquasilineare1}
Annamaria Canino.
\newblock Multiplicity of solutions for quasilinear elliptic equations.
\newblock {\em Topol. Methods Nonlinear Anal.}, 6(2):357--370, 1995.

\bibitem{caninoserdica}
Annamaria Canino.
\newblock On a variational approach to some quasilinear problems.
\newblock {\em Serdica Math. J.}, 22(3):297--324, 1996.


\bibitem{nonsmooththeory1}
Annamaria Canino and Marco Degiovanni.
\newblock Nonsmooth critical point theory and quasilinear elliptic equations.
\newblock In {\em Topological methods in differential equations and inclusions
  ({M}ontreal, {PQ}, 1994)}, volume 472 of {\em NATO Adv. Sci. Inst. Ser. C:
  Math. Phys. Sci.}, pages 1--50. Kluwer Acad. Publ., Dordrecht, 1995.
  
 \bibitem{CaninoMauroNeumann2026}
Annamaria Canino and Simone Mauro. Quasilinear Equations with Neumann Boundary Conditions. \emph{Mediterr. J. Math.} 23, 44 (2026). 



  \bibitem{canino2025vectorial}
Annamaria Canino and Simone Mauro.
\newblock Existence results for variational quasilinear elliptic systems involving the vectorial $p$-Laplacian.
\newblock {\em arXiv preprint arXiv:2510.15694}, 2025.


\bibitem{carmona2013regularity}
Jos{\'e} Carmona, Silvia Cingolani, Pedro~J. Mart{\'\i}nez-Aparicio, and
  Giuseppina Vannella.
\newblock Regularity and morse index of the solutions to critical quasilinear
  elliptic systems.
\newblock {\em Communications in Partial Differential Equations},
  38(10):1675--1711, 2013.


\bibitem{benedetto1989boundary}
Ya-Zhe Chen and Emmanuelle Di~Benedetto.
\newblock Boundary estimates for solutions of nonlinear degenerate parabolic
  systems.
\newblock {\em Journal für die reine und angewandte Mathematik}, 1989.




\bibitem{CianchiArma2014}
Andrea Cianchi and Vladimir~G. Maz'ya.
\newblock Global boundedness of the gradient for a class of nonlinear elliptic
  systems.
\newblock {\em Arch. Ration. Mech. Anal.}, 212(1):129--177, 2014.

\bibitem{cianchi2019optimal}
Andrea Cianchi and Vladimir~G. Maz'ya.
\newblock Optimal second-order regularity for the p-Laplace system.
\newblock {\em Journal de Math{\'e}matiques Pures et Appliqu{\'e}es},
  132:41--78, 2019.

\bibitem{corvellec1993deformation}
Jean-No\"el Corvellec, Marco Degiovanni, and Marco Marzocchi.
\newblock Deformation properties for continuous functionals and critical point
  theory.
\newblock {\em Topol. Methods Nonlinear Anal.}, 1(1):151--171, 1993.

\bibitem{GradientConvergence98}
Gianni Dal Maso, François Murat. Almost everywhere convergence of gradients of solutions to
nonlinear elliptic systems. \emph{Nonlinear Analysis} 31 (1998), 405–412.


\bibitem{degiovanni1994critical}
Marco Degiovanni and Marco Marzocchi.
\newblock A critical point theory for nonsmooth functional.
\newblock {\em Annali di Matematica Pura ed Applicata}, 167:73--100, 1994.


\bibitem{guedda1989quasilinear}
Mohammed Guedda and Laurent V{\'e}ron.
\newblock Quasilinear elliptic equations involving critical Sobolev exponents.
\newblock {\em Nonlinear Anal. Theory Methods Applic.}, 13(8):879--902, 1989.

\bibitem{MingioneSteinTheorem}
Tuomo Kuusi, and Giuseppe Mingione. 
\newblock A nonlinear Stein theorem.
\newblock{\em Calculus of Variations and Partial Differential Equations}, 51.1 (2014): 45-86.

\bibitem{ladyzhenskaya1968linear}
Olga A. Ladyzhenskaya and Nina N. Ural'tseva, 
\textit{Linear and Quasi-linear Elliptic Equations}, 
Academic Press, New York, 1968.


\bibitem{montoro2025regularity}
Luigi Montoro, Luigi Muglia, Berardino Sciunzi, and Domenico Vuono.
\newblock Regularity and symmetry results for the vectorial $p$-Laplacian.
\newblock {\em Nonlinear Analysis}, 251:113700, 2025.

   \bibitem{pellacci1997critical}
Benedetta Pellacci.
\newblock Critical points for non differentiable functionals.
\newblock {\em Bollettino dell'Unione Matematica Italiana B}, 7(11):733--749,
  1997.

  \bibitem{SchmidtMinimizer14}
Thomas Schmidt.
\newblock Partial regularity for degenerate variational problems and image
  restoration models in {BV}.
\newblock {\em Indiana Univ. Math. J.}, 63(1):213--279, 2014.

\bibitem{sciunzi2025global}
Berardino Sciunzi, Giuseppe Spadaro, and Domenico Vuono.
\newblock Global second order optimal regularity for the vectorial
  $p$-Laplacian.
\newblock {\em arXiv preprint arXiv:2502.17067}, 2025.

\bibitem{SquassinaDiagonalSystems99}
Marco Squassina.
\newblock Existence of multiple solutions for quasilinear diagonal elliptic systems.
\newblock Electronic Journal of Differential Equations, Vol. \textbf{1999} (1999). 


\bibitem{squassina2009existence}
Marco Squassina.
\newblock Existence, multiplicity, perturbation, and concentration results for
  a class of quasi-linear elliptic problems.
\newblock {\em Electronic Journal of Differential Equations}, pages 07--213,
  2009.


  \bibitem{struwe1983quasilinear}
Michael Struwe.
\newblock Quasilinear elliptic eigenvalue problems.
\newblock {\em Commentarii Mathematici Helvetici}, 58(1):509--527, 1983.



\end{thebibliography}
\end{document}